\documentclass[11pt]{article}
\usepackage{fullpage}
\usepackage[utf8]{inputenc} 
\usepackage{hyperref}       
\usepackage{url}            
\usepackage{booktabs}       
\usepackage{amsfonts}       
\usepackage{nicefrac}       
\usepackage{microtype}      

\usepackage{amsmath}
\usepackage{graphicx, caption, subcaption}
\usepackage{amsthm}
\usepackage[ruled,linesnumbered]{algorithm2e}
\usepackage{xcolor}
\newtheorem{theorem}{Theorem}[section]
\newtheorem{corollary}[theorem]{Corollary}
\newtheorem{lemma}[theorem]{Lemma}
\newtheorem{proposition}[theorem]{Proposition}

\newtheorem{definition}[theorem]{Definition}
\newtheorem{remark}[theorem]{Remark}
\newcommand{\RR}{\mathbb{R}}

\newcommand{\Binner}{B_{\text{inner}}}
\newcommand{\Bouter}{B_{\text{outer}}}

\newcommand{\prox}{\mathrm{prox}}

\begin{document}
	\title{The Landscape of the Proximal Point Method for Nonconvex-Nonconcave Minimax Optimization}
	\author{Benjamin Grimmer\footnote{bdg79@cornell.edu; Google Research, New York NY and Cornell University, Ithaca NY}, Haihao Lu\footnote{Haihao.Lu@chicagobooth.edu; Google Research, New York NY and University of Chicago, Chicago IL}, Pratik Worah\footnote{pworah@google.com; Google Research, New York NY}, Vahab Mirrokni\footnote{mirrokni@google.com;Google Research, New York NY}}
	\date{}
	\maketitle
	
	\begin{abstract}
		Minimax optimization has become a central tool in machine learning with applications in robust optimization, reinforcement learning, GANs, etc. These applications are often nonconvex-nonconcave, but the existing theory is unable to identify and deal with the fundamental difficulties this poses. In this paper, we study the classic proximal point method (PPM) applied to nonconvex-nonconcave minimax problems. We find that a classic generalization of the Moreau envelope by Attouch and Wets provides key insights. Critcally, we show this envelope not only smooths the objective but can convexify and concavify it based on the level of interaction present between the minimizing and maximizing variables. From this, we identify three distinct regions of nonconvex-nonconcave problems. When interaction is sufficiently strong, we derive global linear convergence guarantees. Conversely when the interaction is fairly weak, we derive local linear convergence guarantees with a proper initialization. Between these two settings, we show that PPM may diverge or converge to a limit cycle.
	\end{abstract}
	
\section{Introduction}
Minimax optimization has become a central tool for modern machine learning, recently receiving increasing attention in optimization and machine learning communities. We consider the following saddle point optimization problem
\begin{equation}\label{eq:main-problem}
\min_{x\in\RR^n}\max_{y\in\RR^m} L(x,y) \ ,
\end{equation}
where $L(x,y)$ is a differentiable function in $x$ and $y$. 
Many important problems in modern machine learning take this form but often have an objective $L(x,y)$ that is neither convex in $x$ nor concave in $y$. For example, 

\begin{itemize}
	\item \textbf{(GANs).} Generative adversarial networks (GANs)~\cite{Goodfellow1412} learn the distribution of observed samples through a two-player zero-sum game. While the generative network (parameterized by $G$) generates new samples, the discriminative network (parameterized by $D$) attempts to distinguish these from the true data. This gives rise to the minimax formulation
	\begin{align*} 
	\min_G \max_D\ &\mathbb{E}_{s\sim p_{data}} \left[\log D(s)\right]  + \mathbb{E}_{e\sim p_{latent}} \left[\log(1-D(G(e)))\right] \ ,
	\end{align*} 
	where $p_{data}$ is the data distribution, and $p_{latent}$ is the latent distribution.
	\item \textbf{(Robust Training).} Minimax optimization has a long history in robust optimization. Recently, it has found usage with neural networks, which have shown great success in machine learning tasks but are vulnerable to adversarial attack. Robust training~\cite{madry2017towards} aims to overcome such issues by solving a minimax problem with adversarial corruptions $y$
	\begin{align*} 
	\min_x \ & \mathbb{E}_{(u, v)} \left[ \max_{y\in S} \ell(u+y, v, x) \right] \ .
	\end{align*} 
	\item \textbf{(Reinforcement Learning).} In reinforcement learning, the solution to Bellman equations can be obtained by solving a primal-dual minimax formulation. Here a dual critic seeks a solution to the Bellman equation and a primal actor seeks state-action pairs to break this satisfaction~\cite{sutton2018reinforcement,Dai1801}.
\end{itemize}

The Proximal Point Method (PPM) may be the most classic first-order method for solving minimax problems. It was first studied in the seminal work by Rockafellar in \cite{rockafellar1976monotone}, and many practical algorithms for minimax optimization developed later on turn out to be approximations of PPM, such as Extragradient Method (EGM)~\cite{tseng1995linear,nemirovski2004prox} and Optimistic Gradient Descent Ascent \cite{DaskalakisNIPS2018}. The update rule of PPM with step-size $\eta$ is given by the proximal operator:
\begin{equation}\label{eq:prox-definition}
(x_{k+1},y_{k+1})=\prox_\eta(x_k,y_k):= \mathrm{arg}\min_{u\in\mathbb{R}^n}\max_{v\in\mathbb{R}^m} L(u,v) +  \frac{\eta}{2}\|u-x_k\|^2 - \frac{\eta}{2}\|v-y_k\|^2.
\end{equation}
For convex-concave minimax problems, PPM is guaranteed to converge to an optimal solution. However, the dynamics of PPM for nonconvex-nonconcave minimax problems are much more complicated. For example, consider the special case of minimax optimization problems with bilinear interaction
\begin{equation}\label{eq:bilinear}
\min_x\max_y  f(x) +x^TAy - g(y).
\end{equation}
Figure \ref{fig:sample-path} presents the sample paths of PPM from different initial solutions solving a simple instance of~\eqref{eq:bilinear} with different interaction terms $A$. This example may be the simplest non-trivial nonconvex-nonconcave minimax problem. It turns out the behaviors of PPM heavily relies on the scale of the interaction term $A$: when the interaction term is small, PPM converges to local stationary solutions, as the interaction term increases, PPM may fall into a limit cycle indefinitely, and eventually when the interaction term is large enough, PPM converges globally to a stationary solution. Similar behaviors also happen in other classic algorithms for nonconvex-nonconcave minimax problems, in particular, EGM, which is known as one of the most effective algorithms for minimax problems. See Figure~\ref{fig:path-more} in Appendix \ref{app:sample-paths} for their trajectories for solving this simple two-dimension example (the study of these other algorithms is beyond the scope of this paper).
In practice, it is also well-known that classic first-order methods may fail to converge to a stable solution for minimax problems, such as GANs~\cite{farnia2020gans}.

\begin{figure}
	\begin{subfigure}[b]{0.32\textwidth}
		\includegraphics[width=\textwidth]{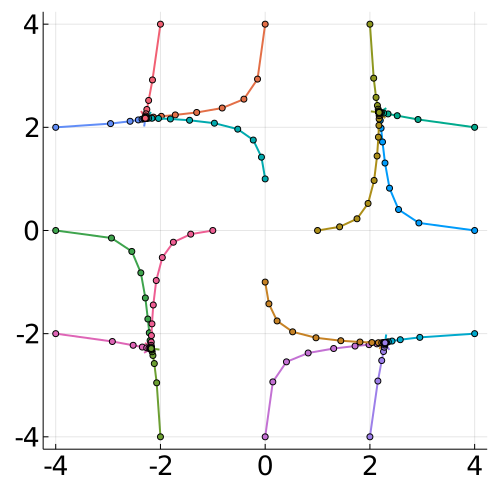}
		\caption{$A=1$}
	\end{subfigure}
	\begin{subfigure}[b]{0.32\textwidth}
		\includegraphics[width=\textwidth]{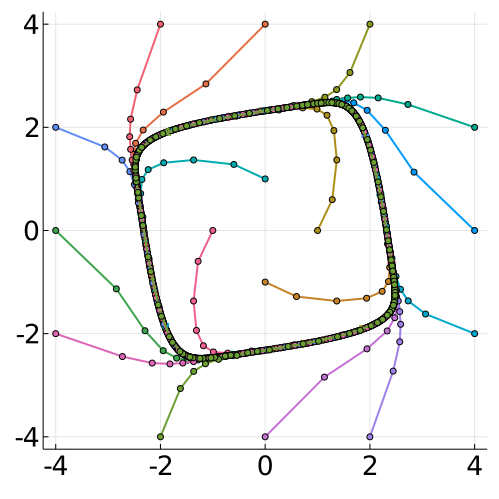}
		\caption{$A=10$}
	\end{subfigure}
	\begin{subfigure}[b]{0.32\textwidth}
		\includegraphics[width=\textwidth]{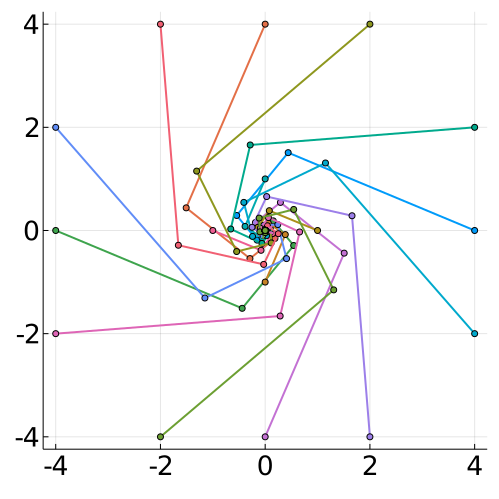}
		\caption{$A=100$}
	\end{subfigure}
	\caption{Sample paths of PPM from different initial solutions applied to \eqref{eq:bilinear} with $f(x)=(x+3)(x+1)(x-1)(x-3)$ and $g(y)=(y+3)(y+1)(y-1)(y-3)$ and different scalars $A$. As $A\ge 0$ increases, the solution path transitions from having four locally attractive stationary points, to a globally attractive cycle, and finally to a globally attractive stationary point.}\label{fig:sample-path}
\end{figure}

The goal of this paper is to understand these varied behaviors of PPM when solving nonconvex-nonconcave minimax problems. We identify that the {\it saddle envelope}, originating from Attouch and Wets~\cite{Attouch1983}, provides key insights
\begin{equation}\label{eq:saddle-envelope}
L_{\eta}(x,y) := \min_{u\in\mathbb{R}^n}\max_{v\in\mathbb{R}^m} L(u,v) +  \frac{\eta}{2}\|u-x\|^2 - \frac{\eta}{2}\|v-y\|^2.
\end{equation}
This generalizes the Moreau envelope but differs in key ways. Most outstandingly, we show that the saddle envelope not only smooths the objective but also can convexify and concavify nonconvex-nonconcave problems when $\nabla^2_{xy} L$ is sufficiently large (which can be interpreted as having a high level of the interaction between $x$ and $y$).
Understanding this envelope in our nonconvex-nonconcave setting turns out to be the cornerstone of explaining the above varied behaviors of PPM.  Utilizing this machinery, we find that the three regions shown in the simple example (Figure \ref{fig:sample-path}) happen with generality for solving \eqref{eq:main-problem}:
\begin{enumerate}
	\item When the interaction between $x$ and $y$ is dominant, PPM has global linear convergence to a stationary point of $L(x,y)$ (Figure \ref{fig:sample-path} (c)). This argument utilizes the fact that, in this case, the closely related saddle envelope becomes convex-concave, even though the $L(x,y)$ is nonconvex-nonconcave.
	\item When the interaction between $x$ and $y$ is weak, properly initializing PPM yields local linear convergence to a nearby stationary point of $L(x,y)$ (Figure \ref{fig:sample-path} (a)). The intuition is that due to the low interaction we do not lose much by ignoring the interaction and decomposing the minimax problem into nonconvex minimization and nonconcave maximization.
	\item Between these interaction dominant and weak regimes, PPM may fail to converge at all and fall into cycling (Figure \ref{fig:sample-path} (b)) or divergence (see the example in Section \ref{subsec:tight}). We construct diverging examples showing our interaction dominant boundary is tight and present a ``Lyapunov''-type function that characterizes how quickly PPM can diverge.
\end{enumerate}
Furthermore, we believe the calculus for the saddle envelope of nonconvex-nonconcave functions that we develop in Section~\ref{sec:saddle-envelope} will be broadly impactful outside its use herein analyzing the proximal point method. As a byproduct of our analysis of the saddle envelope, we clearly see that the interaction term helps the convergence of PPM for minimax problems. This may not be the case for other algorithms, such as gradient descent ascent (GDA) and alternating gradient descent ascent (AGDA) (see Appendix \ref{app:sample-paths} for examples and~\cite{grimmer2020limiting} for theoretical analysis). 

We comment on the meaning of stationary points $\nabla L(z)=0$ for nonconvex-nonconcave problems.
By viewing the problem~\eqref{eq:main-problem} as a simultaneous zero-sum game between a player selecting $x$ and a player selecting $y$, a stationary point can be thought of as a first-order Nash Equilibrium. That is, neither player tends to deviate from their position based on their first-order information.
Beyond the scope of this work, one could instead view~\eqref{eq:main-problem} as a sequential zero-sum game. 
Then a different asymmetric measure of optimality may be called for~\cite{DaskalakisNIPS2018,Jordan1902,farnia2020gans}.

In the rest of this section, we discuss the assumptions, related literature, and preliminaries that will be used later on. In Section \ref{sec:saddle-envelope}, we develop our expanded theory for the saddle envelope. In particular, we introduce the interaction dominance condition (Definition \ref{def:interaction-dominance}) that naturally comes out as a condition for convexity-concavity of the saddle envelope. In Section \ref{sec:interaction-dominate}, we present the global linear convergence of PPM for solving interaction dominant minimax problems. In Section \ref{sec:interaction-weak}, we show that in an interaction weak setting, PPM converges to local stationary points. In Section \ref{sec:interaction-moderate}, we show that PPM may diverge when our interaction dominance condition is slightly violated, establishing the tightness of our global convergence theory. Further, we propose a natural ``Lyapunov''-type function applying to generic minimax problems and providing a bound on how quickly PPM can diverge in the difficult interaction moderate setting.

\subsection{Assumptions and Algorithms}
We say a  function $M(x,y)$ is $\beta$-smooth if its gradient is uniformly $\beta$-Lipschitz
$ \|\nabla M(z) - \nabla M(z')\| \leq \beta\|z-z'\| $
or equivalently for twice differentiable functions, if $\|\nabla^2 M(z)\| \leq \beta$.
Further, we say a twice differentiable $M(x,y)$ is $\mu$-strongly convex-strongly concave for some $\mu\geq 0$ if
$ \nabla^2_{xx} M(z) \succeq \mu I $ and $-\nabla^2_{yy}M(z) \succeq \mu I.$
When $\mu=0$, this corresponds to $M$ being convex with respect to $x$ and concave with respect to $y$.

Throughout this paper, we are primarily interested in the weakening of this convexity condition to allow negative curvature given by $\rho$-weak convexity and $\rho$-weak concavity: we assume that $L$ is twice differentiable, and for any $z=(x,y)\in\RR^n\times\RR^m$ that
\begin{equation} \label{eq:weak}
\nabla^2_{xx} L(z)\succeq -\rho I\ ,\ \ \  -\nabla^2_{yy} L(z)\succeq -\rho I \ .
\end{equation}
Notice that the objective $L(x,y)$ is convex-concave when $\rho=0$, and strongly convex-strongly concave when $\rho<0$. Here our primary interest is in the regime where $\rho>0$ is positive, quantifying how nonconvex-nonconcave $L$ is.

Besides PPM, Gradient Descent Ascent (GDA) is another classic algorithm for minimax problems~\eqref{eq:main-problem} given by
\begin{equation}\label{eq:GDA}
\begin{bmatrix} x_{k+1} \\ y_{k+1} \end{bmatrix} = \begin{bmatrix} x_{k} \\ y_{k} \end{bmatrix} - s\begin{bmatrix} \nabla_x L(x_k,y_k) \\ -\nabla_y L(x_k,y_k) \end{bmatrix} \ ,
\end{equation}
with stepsize parameter $s>0$. However, GDA is known to work only for strongly convex-strongly concave minimax problems, and it may diverge even for simple convex-concave problems~\cite{DaskalakisNIPS2018,lu2020s}.

In this paper, we study a more generalized algorithm, damped PPM, with damping parameter $\lambda\in(0,1]$ and proximal parameter $\eta>0$ given by
\begin{equation}\label{eq:saddle-PPM}
\begin{bmatrix} x_{k+1} \\ y_{k+1} \end{bmatrix} = (1-\lambda)\begin{bmatrix} x_{k} \\ y_{k} \end{bmatrix} +\lambda\ \prox_\eta(x_k,y_k) \ .
\end{equation}
In particular, when $\lambda=1$, we recover the traditional PPM \eqref{eq:prox-definition}.
Interestingly, we find through our theory that some nonconvex-nonconcave problems only have PPM converge when damping is employed (i.e., $\lambda <1$).

\subsection{Related Literature.}

There is a long history of research into convex-concave minimax optimization. Rockafellar \cite{rockafellar1976monotone} studies PPM for solving monotone variational inequalities, and shows that, as a special case, PPM  converges to the stationary point linearly when $L(x,y)$ is strongly convex-strongly concave or when $L(x,y)$ is bilinear. Later on, Tseng~\cite{tseng1995linear} shows that EGM converges linearly to a stationary point under similar conditions. Nemirovski~\cite{nemirovski2004prox} shows that EGM approximates PPM and presents the sublinear rate of EGM. Recently, minimax problems have gained the attention of the machine learning community, perhaps due to the thriving of research on GANs. Daskalakis and Panageas~\cite{DaskalakisNIPS2018} present an Optimistic Gradient Descent Ascent algorithm (OGDA) and shows that it converges linearly to the saddle-point when $L(x,y)$ is bilinear. Mokhtari et al.~\cite{mokhtari2019unified} show that OGDA is a different approximation to PPM. Lu~\cite{lu2020s} presents an ODE approach, which leads to unified conditions under which each algorithm converges, including a class of nonconvex-nonconcave problems.

There are also extensive studies on convex-concave minimax problems when the interaction is bilinear (similar to our example \eqref{eq:bilinear}). Some influential algorithms include Nesterov's smoothing~\cite{nesterov2005smooth}, Douglas-Rachford splitting (a special case is Alternating Direction Method of Multipliers (ADMM))~\cite{douglas1956numerical,eckstein1992douglas} and Primal-Dual Hybrid Gradient Method (PDHG)~\cite{chambolle2011first}.

Recently, a number of works have been undertaken considering nonconvex-concave minimax problems. The basic technique is to turn the minimax problem \eqref{eq:main-problem} to a minimization problem on $\Phi(x) = \max_y L(x,y)$, which is well-defined since $L(x,y)$ is concave in $y$, and then utilize the recent developments in nonconvex optimization~\cite{Jordan1906,Jordan2002,Rafique1810,Thekumparampil1912}.

Unfortunately, the above technique cannot be extended to the nonconvex-nonconcave setting, because $\Phi(x)$ is now longer tractable to compute (even approximately) as it is a nonconcave maximization problem itself. Indeed, the current understanding of nonconvex-nonconcave minimax problems is fairly limited. The recent research on nonconvex-nonconcave minimax problems mostly relies on some form of convex-concave-like assumptions, based on Stampacchia's or Minty's Variational Inequality~\cite{Zhou2017,Lin1809,Diakonikolas2021} and Polyak-Lojasiewicz conditions~\cite{Nouiehed1902,Yang2002}, which are strong in general and successfully bypass the inherent difficulty in the nonconvex-nonconcave setting. Such theory, unfortunately, presupposes the existence of a globally attractive solution and thus cannot describe behaviors like local solutions and cycling.

In an early version of this work~\cite{grimmer2020landscape}, we presented preliminary results for analyzing nonconvex-nonconcave bilinear problem \eqref{eq:bilinear}. Simultaneous to (or after) the early version, \cite{letcher2020impossibility} presents examples of nonconvex-nonconcave minimax problems where a reasonably large class of algorithms do not converge; \cite{Hsieh2006} presents an ODE analysis for the limiting behaviors of different algorithms with step-size shrinking to zero and shows the possibility to converge to an attractive circle; \cite{zhang2020optimality} utilizes tools from discrete-time dynamic systems to study the behaviors of algorithms around a local stationary solution, which involves the complex eigenvalues of a related Jacobian matrix;
\cite{grimmer2020limiting} studies the phase transitions between limit cycles and limit points of higher-order resolution ODEs of different nonconvex-nonconcave algorithms. 
Compared to these works, we identify machinery in the saddle envelope that facilitates direct analysis.

Epi/hypo-convergence of saddle functions and in particular the saddle envelope is developed by Attouch and Wets~\cite{Attouch1983,Attouch1983-second}. These notions of convergence facilitate asymptotic studies of penalty methods~\cite{Flam1986} and approximate saddle points~\cite{Guillerme1989}. Rockafellar~\cite{Rockafellar1985,Rockafellar1990} further builds generalized second derivatives for saddle functions, which may provide an avenue to relax our assumptions of twice differentiability here.
Assuming the given function $L$ is convex-concave, continuity/differentiability properties and relationships between saddle points are developed in~\cite{Attouch1986,Aze1988}, which facilitate asymptotic convergence analysis of proximal point methods like~\cite{Mouallif1989}. In Section~\ref{sec:saddle-envelope}, we build on these results, giving a calculus for the saddle envelope of nonconvex-nonconcave functions.

The idea to utilize a generalization of the Moreau envelope for nonconvex-nonconcave minimax problems is well motivated by the nonconvex optimization literature. In recent years, the Moreau envelope has found great success as an analysis tool in nonsmooth nonconvex optimization~\cite{Davis1901,Davis1803,Zhang1806} and in nonconvex-concave optimization~\cite{Rafique1810}.

\subsection{Preliminaries} \label{sec:prelim}
Strongly convex-strongly concave optimization problems $\min_x\max_y M(x,y)$ are well understood. The following lemma and subsequent theorem show that GDA contracts towards a stationary point when strong convexity-strong concavity and smoothness hold locally. Proofs are given in Appendix~\ref{app:proofs} for completeness. 
\begin{lemma} \label{lem:helper0}
	Suppose $M(x,y)$ is $\mu$-strongly convex-strongly concave on a convex set $S=S_x\times S_y$, then it holds for any $(x,y),(x',y') \in S$ that
	$$\mu\left\|\begin{bmatrix} x-x' \\ y-y' \end{bmatrix}\right\|^2 \leq \left(\begin{bmatrix} \nabla_x M(x,y)  \\-\nabla_y M(x,y) \end{bmatrix} -\begin{bmatrix} \nabla_x M(x',y')  \\-\nabla_y M(x',y') \end{bmatrix}\right)^T\begin{bmatrix} x-x' \\ y-y' \end{bmatrix}.$$
	When $\nabla M(x', y')=0$, the distance to this stationary point is bounded by
	$$
	\left\|\begin{bmatrix} x-x' \\ y-y' \end{bmatrix}\right\| \leq \frac{\|\nabla M(x,y)\|}{\mu} \ . 
	$$
\end{lemma}
\begin{theorem} \label{thm:standard-contraction}
	Consider any minimax problem $\min_{x\in\RR^n}\max_{y\in\RR^m} M(x,y)$ where $M(x,y)$ is $\beta$-smooth and $\mu$-strongly convex-strongly concave on a set $B(x_0,r)\times B(y_0,r)$ with $r\geq 2\|\nabla M(x_0,y_0)\|/\mu$. Then GDA~\eqref{eq:GDA} with step-size $s\in(0, 2\mu/\beta^2) $ linearly converges to a stationary point $(x^*,y^*)\in B((x_0,y_0),r/2)$ with
	$$ \left\|\begin{bmatrix} x_k-x^* \\ y_k-y^* \end{bmatrix}\right\|^2 \leq \left(1 - 2\mu s+\beta^2s^2\right)^{k}\left\|\begin{bmatrix} x_0-x^* \\ y_0-y^* \end{bmatrix}\right\|^2.$$
\end{theorem}

We denote the Moreau envelope of a function $f$ with parameter $\eta>0$ by
\begin{equation}
e_{\eta}\left\{f\right\}(x) = \min_{u}\ f(u) + \frac{\eta}{2}\|u - x\|^2 \ . \label{eq:moreau-def}
\end{equation}
The Moreau envelope of a function provides a lower bound on it everywhere as
\begin{equation}
e_\eta\{f\}(x) \leq f(x) \ . \label{eq:moreau-lower-bound} 
\end{equation}
For $\rho$-weakly convex functions and $\eta>\rho$, there is a nice calculus for the Moreau envelope.
Its gradient at some $x\in\RR^n$ is determined by the proximal step $x_+=\mathrm{argmin}_u\  f(u) + \frac{\eta}{2}\|u - x\|^2$ having
\begin{equation}
\nabla e_{\eta}\left\{f\right\}(x) = \eta(x-x_+) = \nabla f(x_+) \ . \label{eq:moreau-grad}
\end{equation}
For twice differentiable $f$, the Moreau envelope has Hessian
\begin{equation}
\nabla^2 e_\eta\{f\}(x) = \eta I - (\eta I + \nabla^2 f(x_+))^{-1} \ . \label{eq:moreau-hess}
\end{equation}
This formula bounds the smoothness and convexity of the envelope by
\begin{equation}
(\eta^{-1} - \rho^{-1})^{-1}I \preceq \nabla^2 e_\eta\{f\}(x) \preceq \eta I \ . \label{eq:moreau-hessian-bounds}
\end{equation}
These bounds ensure the Moreau envelope has a $\max\{\eta,|\eta^{-1}-\rho^{-1}|^{-1}\}$-Lipschitz gradient, which simplifies for convex $f$ (that is, $\rho\leq 0$) to have an $\eta$-Lipschitz gradient.
Noting that $(\eta^{-1} - \rho^{-1})^{-1}$ always has the same sign as $-\rho$, we see that the Moreau envelope is (strongly/weakly) convex exactly when the given function $f$ is (strongly/weakly) convex.

\section{The Saddle Envelope} \label{sec:saddle-envelope}
In this section, we consider the saddle envelope first developed by Attouch and Wets~\cite{Attouch1983} and characterize its structure for nonconvex-nonconcave optimization. Recall for any proximal parameter $\eta>0$, the saddle envelope (also referred to as an upper Yosida approximate and a mixed Moreau envelope) is defined as
\begin{equation*}
L_{\eta}(x,y) := \min_{u\in\mathbb{R}^n}\max_{v\in\mathbb{R}^m} L(u,v) +  \frac{\eta}{2}\|u-x\|^2 - \frac{\eta}{2}\|v-y\|^2.
\end{equation*}
We require that the parameter $\eta$ is selected with $\eta > \rho$, which ensures the minimax problem in~\eqref{eq:saddle-envelope} is strongly convex-strongly concave. As a result, the saddle envelope is well-defined (as its subproblem has a unique minimax point) and often can be efficiently approximated.

The saddle envelope generalizes the Moreau envelope from the minimization literature to minimax problems. To see this reduction, taking any objective $L(x,y)=g(x)$ (that is, one constant with respect to $y$) gives
$$ L_{\eta}(x,y) = \min_u \max_v\ g(u) + \frac{\eta}{2}\|u-x\|^2 - \frac{\eta}{2}\|v-y\|^2 = e_\eta\{g\}(x) \ .$$
We take careful note throughout our theory of similarities and differences from the simpler case of Moreau envelopes. We begin by considering how the value of the saddle envelope $L_\eta$ relates to the origin objective $L$.
Unlike the Moreau envelope in~\eqref{eq:moreau-lower-bound}, the saddle envelope fails to provide a lower bound. If the objective function is constant with respect to $y$, having $L(x,y)=g(x)$, the saddle envelope becomes a Moreau envelope and provides a lower bound for every $(x,y)$,
$$ L_{\eta}(x,y) = e_\eta\{g\}(x) \leq g(x)=L(x,y) \ .$$
Conversely, if $L(x,y)=h(y)$, then the saddle envelope provides an upper bound.
In generic settings between these extremes, the saddle envelope $L_\eta$ need not provide any kind of bound on $L$. The only generic relationship we can establish between $L(z)$ and $L_\eta(z)$ everywhere is that as $\eta\rightarrow \infty$, they approach each other. This result is formalized by~\cite{Attouch1983} through epi-hypo convergence.

In the following pair of subsections, we build on the classic results of~\cite{Attouch1983,Attouch1986,Aze1988} by deriving a calculus of the saddle envelope of nonconvex-nonconcave functions (in Section~\ref{subsec:saddle2}) and characterizing the smoothing and convexifying effects of this operation (in Section~\ref{subsec:saddle3}). 

\subsection{Calculus for the Saddle Envelope $L_\eta$} \label{subsec:saddle2}
Here we develop a calculus for the saddle envelope $L_\eta$, giving formulas for its gradient and Hessian in terms of the original objective $L$ and the proximal operator. These results immediately give algorithmic insights into the proximal point method.
First, we show that a generalization of the Moreau gradient formula~\eqref{eq:moreau-grad} and the convex-concave formula of~\cite[Theorem 5.1 (d)]{Attouch1986} holds.
\begin{lemma} \label{lem:grad-formula}
	The gradient of the saddle envelope $L_\eta(x,y)$ at $z=(x,y)$ is
	$$ \begin{bmatrix} \nabla_x L_\eta(z)\\  \nabla_y L_\eta(z)\end{bmatrix} = \begin{bmatrix} \eta(x-x_+)\\ \eta(y_+-y) \end{bmatrix} = \begin{bmatrix} \nabla_x L(z_+)\\  \nabla_y L(z_+)\end{bmatrix}$$
	where $z_+ = (x_+,y_+) = \prox_\eta(z)$ is given by the proximal operator.
\end{lemma}
\begin{proof}
	Notice that the saddle envelope is a composition of Moreau envelopes
	\begin{align*}
	L_{\eta}(x,y) &= \min_{u} \left(\max_{v} L(u,v) - \frac{\eta}{2}\|v-y\|^2\right) + \frac{\eta}{2}\|u-x\|^2\\
	&= \min_{u}\ -e_\eta\{-L(u,\cdot)\}(y) + \frac{\eta}{2}\|u-x\|^2 = e_{\eta}\{g(\cdot,y)\}(x)
	\end{align*}
	where $g(u,y) = -e_{\eta}\{-L(u,\cdot)\}(y)$. Applying the gradient formula~\eqref{eq:moreau-grad} gives our first claimed gradient formula in $x$ of
	$\nabla_x L_\eta(x,y) = \eta (x-x_+)$
	since $x_+$ is the unique minimizer of $u \mapsto g(u,y) + \frac{\eta}{2}\|u-x\|^2$. Symmetric reasoning gives our first claimed formula $\nabla_y L_\eta(x,y) = \eta (y_+-y) $ in $y$.
	The second claimed equality is precisely the first-order optimality condition for~\eqref{eq:prox-definition}. 
\end{proof}
\begin{corollary} \label{cor:stationarity}
	The stationary points of $L_\eta$ are exactly the same as those of $L$.
\end{corollary}
\begin{proof}
	First consider any stationary point $z=(x,y)$ of $L$. Denote $z_+=\prox_\eta(z)$ and the objective function defining the proximal operator~\eqref{eq:prox-definition} as $M(u,v) = L(u,v) + \frac{\eta}{2}\|u-x\|^2-\frac{\eta}{2}\|v-y\|^2$. Then observing that $\nabla M(z)=0$, $z$ must be the unique minimax point of $M$ (that is, $z=z_+=\prox_\eta(z)$). Hence $z$ must be a stationary point of $L_\eta$ as well since $\nabla L_\eta(z) =\nabla L(z_+)=\nabla L(z)= 0$.
	
	Conversely consider a stationary point $z=(x,y)$ of $L_\eta$. Then $\eta(x-x_+) = \nabla_x L_\eta(z)=0$ and $\eta(y_+-y) = \nabla_y L_\eta(z)=0$. Hence we again find that $z=z_+=\prox_\eta(z)$ and consequently, this point must be a stationary point of $L$ as well since $\nabla L(z)=\nabla L(z_+)=\nabla L_\eta(z)=0$.
\end{proof}
\begin{corollary} \label{cor:PPM-to-GDA}
	One step of the (damped) PPM~\eqref{eq:saddle-PPM} on the original objective $L$ is equivalent to one step of GDA~\eqref{eq:GDA} on the saddle envelope $L_\eta$ with $s = \lambda/\eta$. 
\end{corollary}
\begin{proof}
	Let $(x^+_k,y^+_k) = \prox_\eta(x_k,y_k)$ and let $(x_{k+1}, y_{k+1})$ be a step of GDA on $L_\eta(x,y)$ from $(x_k,y_k)$ with step-size $s=\lambda/\eta$. Then
	\begin{align*}
	\begin{bmatrix} x_{k+1} \\ y_{k+1} \end{bmatrix} = \begin{bmatrix} x_{k} \\ y_{k}\end{bmatrix} - s\begin{bmatrix} \nabla_xL_\eta(z_k) \\ -\nabla_yL_\eta(z_k)\end{bmatrix}
	&= \begin{bmatrix} x_{k} \\ y_{k}\end{bmatrix} - \frac{\lambda}{\eta}\begin{bmatrix} \eta(x_k-x^+_k) \\ -\eta(y^+_k-y_k) \end{bmatrix}
	= (1-\lambda)\begin{bmatrix} x_{k} \\ y_{k}\end{bmatrix} + \lambda\begin{bmatrix} x^+_k \\ y^{+}_k \end{bmatrix}
	\end{align*}
	follows from Lemma~\ref{lem:grad-formula}.
\end{proof}

Similar to our previous lemma, the Hessian of the saddle envelope at some $z$ is determined by the Hessian of $L$ at $z_+=\prox_\eta(z)$. This formula generalizes the Moreau envelope's formula~\eqref{eq:moreau-hess} whenever $L$ is constant with respect to $y$.
\begin{lemma} \label{lem:hess-formula}
	The Hessian of the saddle envelope $L_\eta(z)$ is
	$$ \begin{bmatrix} \nabla^2_{xx} L_\eta(z) & \nabla^2_{xy} L_\eta(z)\\  -\nabla^2_{yx}L_\eta(z) & -\nabla^2_{yy} L_\eta(z)\end{bmatrix} = \eta I-\eta^2\left(\eta I+ \begin{bmatrix} \nabla^2_{xx} L(z_+) & \nabla^2_{xy} L(z_+)\\  -\nabla^2_{yx} L(z_+) & -\nabla^2_{yy} L(z_+)\end{bmatrix}\right)^{-1} $$
	where $z_+ = \prox_\eta(z)$. Since $\eta>\rho$, we have
	$$ \nabla^2_{xx} L_\eta(z) = \eta I -\eta^2\left(\eta I +\nabla_{xx} L(z_+) + \nabla^2_{xy} L(z_+)(\eta I -\nabla^2_{yy} L(z_+))^{-1}\nabla^2_{yx} L(z_+)\right)^{-1},$$
	$$ \nabla^2_{yy} L_\eta(z) = -\eta I +\eta^2\left(\eta I +\nabla_{yy} L(z_+) + \nabla^2_{yx} L(z_+)(\eta I +\nabla^2_{xx} L(z_+))^{-1}\nabla^2_{xy} L(z_+)\right)^{-1}.$$
\end{lemma}
\begin{proof}
	Consider some $z=(x,y)$ and a nearby point $z^\Delta=z+\Delta$. Denote one proximal step from each of these points by $z_+=(x_+,y_+)= \prox_\eta(z)$ and $z^\Delta_+=(x^\Delta_+,y^\Delta_+)= \prox_\eta(z^\Delta)$. 
	Then our claimed Hessian formula amounts to showing
	\begin{align*}
		\begin{bmatrix} \nabla_x L_\eta(z^\Delta) \\ -\nabla_y L_\eta(z^\Delta) \end{bmatrix} - \begin{bmatrix} \nabla_x L_\eta(z) \\ -\nabla_y L_\eta(z) \end{bmatrix}  &= \left(\eta I - \eta^2\left(\eta I + \begin{bmatrix} \nabla^2_{xx} L(z_+) & \nabla^2_{xy} L(z_+)\\  -\nabla^2_{yx} L(z_+) & -\nabla^2_{yy} L(z_+)\end{bmatrix}\right)^{-1}\right)\begin{bmatrix} \Delta_x\\ \Delta_y  \end{bmatrix} +o(\|\Delta\|).
	\end{align*}
	Recall Lemma~\ref{lem:grad-formula} showed the gradient of the saddle envelope is given by $\nabla_x L_\eta(z) = \eta(x-x_+)$ and $\nabla_y L_\eta(z) = \eta(y_+-y)$. Applying this at $z$ and $z_+$ and dividing by $\eta$, our claimed Hessian formula becomes
	\begin{equation} \label{eq:target-hessian-condition}
	\begin{bmatrix} x^\Delta_+-x_+\\ y^\Delta_+-y_+ \end{bmatrix} = \eta\left(\eta I + \begin{bmatrix} \nabla^2_{xx} L(z_+) & \nabla^2_{xy} L(z_+)\\  -\nabla^2_{yx} L(z_+) & -\nabla^2_{yy} L(z_+)\end{bmatrix}\right)^{-1}\begin{bmatrix} \Delta_x\\ \Delta_y \end{bmatrix} +o(\|\Delta\|) \ .
	\end{equation}
	Our proof shows this in two steps: first considering a proximal step on the second-order Taylor approximation of $L$ at $z_+$ and then showing this closely matches the result of a proximal step on $L$.
	
	First, consider the following quadratic model of the objective around $z_+$:
	$$ \widetilde L(z) = L(z_+) + \nabla L(z_+)^T(z-z_+) + \frac{1}{2}(z-z_+)^T\nabla^2 L(z_+)(z-z_+) \ . $$
	Denote the result of one proximal step on $\widetilde L$ from $z^\Delta$ by $\widetilde z^\Delta_+ = (\widetilde x_+^\Delta, \widetilde y_+^\Delta)$. Since the proximal subproblem is strongly convex-strongly concave, this solution is uniquely determined by
	$$
	\begin{bmatrix} \nabla_x \widetilde L(\widetilde x^\Delta_+,\widetilde y^\Delta_+)\\  -\nabla_y \widetilde L(\widetilde x^\Delta_+,\widetilde y^\Delta_+)\end{bmatrix} + \begin{bmatrix} \eta(\widetilde x^\Delta_+-x^\Delta)\\ \eta(\widetilde y^\Delta_+-y^\Delta) \end{bmatrix} = 0 \ .
	$$
	Plugging in the definition of our quadratic model $\widetilde L$ yields
	\begin{align*}
	&\begin{bmatrix} \nabla_x L(z_+)\\ -\nabla_y L(z_+)\end{bmatrix} + \begin{bmatrix} \nabla^2_{xx} L(z_+) & \nabla^2_{xy} L(z_+)\\  -\nabla^2_{yx} L(z_+) & -\nabla^2_{yy} L(z_+)\end{bmatrix}\begin{bmatrix} \widetilde x^\Delta_+-x_+\\ \widetilde y^\Delta_+-y_+ \end{bmatrix} + \begin{bmatrix} \eta(\widetilde x^\Delta_+-x^\Delta)\\ \eta(\widetilde y^\Delta_+-y^\Delta) \end{bmatrix} = 0 \ .
	\end{align*}
	Hence
	\begin{align*}
	& \left(\eta I + \begin{bmatrix} \nabla^2_{xx} L(z_+) & \nabla^2_{xy} L(z_+)\\  -\nabla^2_{yx} L(z_+) & -\nabla^2_{yy} L(z_+)\end{bmatrix}\right)\begin{bmatrix} \widetilde x^\Delta_+-x_+\\  \widetilde y^\Delta_+-y_+ \end{bmatrix} = \eta \begin{bmatrix} x^\Delta - x_+ -\eta^{-1}\nabla_x L(z_+) \\ y^\Delta - y_+ +\eta^{-1}\nabla_y L(z_+) \end{bmatrix}\\
	&\implies  \left(\eta I + \begin{bmatrix} \nabla^2_{xx} L(z_+) & \nabla^2_{xy} L(z_+)\\  -\nabla^2_{yx} L(z_+) & -\nabla^2_{yy} L(z_+)\end{bmatrix}\right)\begin{bmatrix} \widetilde x^\Delta_+-x_+\\  \widetilde y^\Delta_+-y_+ \end{bmatrix} = \eta \begin{bmatrix} \Delta_x\\ \Delta_y \end{bmatrix}\\
	&\implies  \begin{bmatrix} \widetilde x^\Delta_+-x_+\\ \widetilde y^\Delta_+-y_+ \end{bmatrix} = \eta\left(\eta I + \begin{bmatrix} \nabla^2_{xx} L(z_+) & \nabla^2_{xy} L(z_+)\\  -\nabla^2_{yx} L(z_+) & -\nabla^2_{yy} L(z_+)\end{bmatrix}\right)^{-1}\begin{bmatrix} \Delta_x\\ \Delta_y \end{bmatrix}.
	\end{align*}
	
	This is nearly our target condition~\eqref{eq:target-hessian-condition}. All that remains is to show our second-order approximation satisfies $\| z^\Delta_+ - \widetilde z^\Delta_+\| = o(\|\Delta\|)$.
	Denote the proximal subproblem objective by $ M^\Delta (u,v) = L(u,v) + \frac{\eta}{2}\|u-x^\Delta\|^2 - \frac{\eta}{2}\|v-y^\Delta\|^2$ and its approximation by $\widetilde M^\Delta (u,v) = \widetilde L(u,v) + \frac{\eta}{2}\|u-x^\Delta\|^2 - \frac{\eta}{2}\|v-y^\Delta\|^2$.
	Noting that $\|\nabla \widetilde M^\Delta (x_+,y_+)\| = \eta\|\Delta\|$, we can apply Lemma~\ref{lem:helper0} to the $(\eta-\rho)$-strongly convex-strongly concave function $\widetilde M^\Delta$ to bound the distance to its minimax point as
	$$ \|z_+-\widetilde z_+^\Delta\|\leq \frac{\eta}{\eta-\rho}\|\Delta\| \ .$$
	Consequently, we can bound difference in gradients between $L$ and its quadratic model $\widetilde L$ at $\widetilde z^\Delta_+$ by $ \|\nabla L(\widetilde z^\Delta_+) - \nabla \widetilde L(\widetilde z^\Delta_+)\| = o(\|\Delta\|)$. Therefore $\|\nabla M^\Delta(\widetilde z^\Delta_+)\| = o(\|\Delta\|)$ and so applying Lemma~\ref{lem:helper0} to the strongly convex-strongly concave function $M^\Delta$ bounds the distance to its minimax point as $\|z^\Delta_+ - \widetilde z^\Delta_+\| = o(\|\Delta\|)$, which completes our proof.
\end{proof}
A careful understanding of the saddle envelope's Hessian allows us to describe its smoothness and when it is convex-concave. This is carried out in the following section and forms the crucial step in enabling our convergence analysis for nonconvex-nonconcave problems. 

\subsection{Smoothing and Convexifing from the Saddle Envelope} \label{subsec:saddle3}
Recall that the Moreau envelope $e_\eta\{f\}$ serves as a smoothing of any $\rho$-weakly convex function since its Hessian has uniform bounds above and below~\eqref{eq:moreau-hessian-bounds}.
The lower bound on the Moreau envelope's Hessian guarantees it is convex exactly when the given function $f$ is convex (that is, $\rho=0$), and strongly convex if and only if $f$ is strongly convex.

In the convex-concave case~\cite[Proposition 2.2]{Aze1988} established the saddle envelope has $1/\eta$-Lipschitz gradient. Our Hessian formula in Lemma~\ref{lem:hess-formula} allows us to quantify the envelope's smoothness for nonconvex-nonconcave objectives. Outstandingly, we find that the minimax extension of this result is much more powerful than its Moreau counterpart. The saddle envelope will be convex-concave not just when $L$ is convex-concave, but whenever the following interaction dominance condition holds with a nonnegative parameter $\alpha$. 
\begin{definition}\label{def:interaction-dominance}
	A function $L$ is $\alpha$-interaction dominant with respect to $x$ if
	\begin{align}
	\nabla^2_{xx} L(z) + \nabla^2_{xy} L(z)(\eta I -\nabla^2_{yy} L(z))^{-1}\nabla^2_{yx} L(z) &\succeq \alpha I \label{eq:x-dominance}
	\end{align}	   
	and $\alpha$-interaction dominant with respect to $y$ if
	\begin{align}
	-\nabla^2_{yy} L(z) + \nabla^2_{yx} L(z)(\eta I +\nabla^2_{xx} L(z))^{-1}\nabla^2_{xy} L(z) &\succeq \alpha I\ . \label{eq:y-dominance}   
	\end{align}	  
\end{definition}

For any $\rho$-weakly convex-weakly concave function $L$, interaction dominance holds with $\alpha=-\rho$ since the second term in these definitions is always positive semidefinite. As a consequence, any convex-concave function is $\alpha \geq 0$-interaction dominant with respect to both $x$ and $y$. Further, nonconvex-nonconcave functions are interaction dominant with $\alpha\geq 0$ when the second terms above are sufficiently positive definite (hence the name ``interaction dominant'' as the interaction term of the Hessian $\nabla^2_{xy} L(z)$ is dominating any negative curvature in Hessians $\nabla^2_{xx} L(z)$ and $-\nabla^2_{yy} L(z)$). For example, any problem with $\beta$-Lipschitz gradient in $y$ has interaction dominance in $x$ hold with non-negative parameter whenever
$$ \frac{\nabla^2_{xy} L(z)\nabla^2_{yx} L(z)}{\eta+\beta} \succeq -\nabla^2_{xx} L(z) $$
since $\eta I -\nabla^2_{yy} L(z) \preceq (\eta+\beta)I$. Similarly, any problem with $\beta$-Lipschitz gradient in $x$ has interaction dominance in $y$ with a non-negative parameter whenever
$$ \frac{\nabla^2_{yx} L(z)\nabla^2_{xy} L(z)}{\eta+\beta} \succeq \nabla^2_{yy} L(z) \ .$$

The following proposition derives bounds on the Hessian of the saddle envelope showing it is convex in $x$ (concave in $y$) whenever $\alpha\geq 0$-interaction dominance holds in $x$ (in $y$). Further, its Hessian lower bounds ensure that $L_{\eta}$ is  $(\eta^{-1} + \alpha^{-1})^{-1}$-strongly convex in $x$ (strongly concave in $y$) whenever $\alpha>0$-interaction dominance holds in $x$ (in $y$).


\begin{proposition} \label{prop:hessians}
	If the $x$-interaction dominance~\eqref{eq:x-dominance} holds with $\alpha\in\RR$, the saddle envelope is smooth and weakly convex with respect to $x$ 
	$$  (\eta^{-1} + \alpha^{-1})^{-1}I \preceq \nabla^2_{xx}L_\eta(z) \preceq \eta I\ ,$$
	and if the $y$-interaction dominance condition~\eqref{eq:y-dominance} holds with $\alpha\in\RR$, the saddle envelope is smooth and weakly concave with respect to $y$
	$$  (\eta^{-1} + \alpha^{-1})^{-1}I \preceq -\nabla^2_{yy}L_\eta(z) \preceq \eta I\ .$$
\end{proposition}
\begin{proof}
	Recall the formula for the $x$ component of the Hessian given by Lemma~\ref{lem:hess-formula}.
	Then the interaction dominance condition~\eqref{eq:x-dominance} can lower bound this Hessian by
	\begin{align*}
	\nabla^2_{xx} L_\eta(z) &=\eta I -\eta^2\left(\eta I +\nabla_{xx} L(z_+) + \nabla^2_{xy} L(z_+)(\eta I -\nabla^2_{yy} L(z_+))^{-1}\nabla^2_{yx} L(z_+)\right)^{-1}\\
	&\succeq \eta I -\eta^2\left(\eta I +\alpha I\right)^{-1}\\
	&=(\eta - \eta^2/(\eta+\alpha))I\\
	&=(\eta^{-1} + \alpha^{-1})^{-1}I \ .
	\end{align*}
	Note that $\eta I +\nabla^2_{xx} L(z_+)$ is positive definite (since $\eta>\rho$) and $\nabla^2_{xy} L(z_+)(\eta I +\nabla^2_{yy} L(z_+))^{-1}\nabla^2_{yx} L(z_+)$ is positive semidefinite (since its written as a square). Then the inverse of their sum must also be positive definite and consequently $\nabla^2_{xx} L_\eta(z)$ is upper bounded by
	\begin{align*}
	\eta I -\eta^2\left(\eta I +\nabla_{xx} L(z_+) + \nabla^2_{xy} L(z_+)(\eta I +\nabla^2_{yy} L(z_+))^{-1}\nabla^2_{yx} L(z_+)\right)^{-1} &\preceq \eta I \ .
	\end{align*}
	Symmetric reasoning applies to give bounds on $\nabla^2_{yy} L_\eta(z)$.
\end{proof}
\begin{remark}
	Note that our definition of interaction dominance depends on the choice of the proximal parameter $\eta >\rho$. In our convergence theory, we will show that interaction dominance with nonnegative $\alpha>0$ captures when the proximal point method with the same parameter $\eta$ converges.
\end{remark}
\begin{remark}
	Proposition~\ref{prop:hessians} generalizes the Hessian bounds for the Moreau envelope~\eqref{eq:moreau-hessian-bounds} since for any $L(x,y)$ that is constant in $y$, the $\alpha$-interaction dominance condition in $x$ simplifies to simply be $\rho$-weak convexity
	$ \nabla^2_{xx} L(z) + \nabla^2_{xy} L(z)(\eta I -\nabla^2_{yy} L(z))^{-1}\nabla^2_{yx} L(z) = \nabla^2_{xx} L(z) \succeq \alpha I$. Hence this special case has $\alpha=-\rho$.
\end{remark}

In addition to bounding the Hessians of the $x$ and $y$ variables separately, we can also bound the overall smoothness of the saddle envelope. Our next result shows that the saddle envelope maintains the same $\max\{\eta, |\eta^{-1}-\rho^{-1}|^{-1}\}$-smoothing effect as the Moreau envelope~\eqref{eq:moreau-hessian-bounds}.
\begin{proposition}\label{prop:smoothness}
	$L_\eta$ has $\max\{\eta, |\eta^{-1}-\rho^{-1}|^{-1}\}$-Lipschitz gradient.
\end{proposition}
\begin{proof}
	Consider two points $z=(x,y)$ and $\bar z = (\bar x, \bar y)$ and denote one proximal step from each of them by $z_+=(x_+,y_+) = \prox_\eta(z)$ and $\bar z_+ = (\bar x_+,\bar y_+) = \prox_\eta(\bar z)$.
	Define the  $(\eta-\rho)$-strongly convex-strongly concave function underlying the computation of the saddle envelope at $z$ as
	$$ M(u,v) = L(u,v) + \frac{\eta}{2}\|u-x\|^2 -\frac{\eta}{2}\|v-y\|^2. $$
	First we compute the gradient of $M$ at $\bar z_+$ which is given by
	\begin{align*}
	\begin{bmatrix} \nabla_x M(\bar z_+) \\ \nabla_y M(\bar z_+)\end{bmatrix} =  \begin{bmatrix} \nabla_x L(\bar z_+) + \eta(\bar x_+-x) \\ \nabla_y L(\bar z_+) - \eta(\bar y_+-y)\end{bmatrix}
	= \eta \begin{bmatrix} \bar x-x \\ y-\bar y\end{bmatrix}.
	\end{align*}
	Applying Lemma~\ref{lem:helper0}, and noting that $z_+=\prox_\eta(z)$ has $\nabla M(z_+)=0$ yields
	$$ \left\|\begin{bmatrix} \bar x_+-x_+ \\ \bar y_+-y_+ \end{bmatrix}\right\|^2 \leq \frac{\eta}{\eta-\rho} \begin{bmatrix} \bar x- x \\ \bar y- y \end{bmatrix}^T\begin{bmatrix} \bar x_+-x_+ \\ \bar y_+-y_+ \end{bmatrix}.$$
	Recalling the saddle envelope's gradient formula from Lemma~\ref{lem:grad-formula}, we can upper bound the difference between its gradients at $z$ and $\bar z$ by
	\begin{align*}
	\frac{1}{\eta^{2}}\|\nabla L_\eta(z) - \nabla L_\eta(\bar z)\|^2
	&= \left\|\begin{bmatrix} x-x_+ \\ y_+-y \end{bmatrix} - \begin{bmatrix} \bar x-\bar x_+ \\ \bar y_+-\bar y \end{bmatrix}\right\|^2\\
	&=\left\|\begin{bmatrix} x-\bar x \\ y-\bar y \end{bmatrix}\right\|^2 + 2 \begin{bmatrix} x-\bar x \\ y-\bar y \end{bmatrix}^T\begin{bmatrix} \bar x_+- x_+ \\ \bar y_+- y_+ \end{bmatrix} +\left\|\begin{bmatrix} \bar x_+- x_+ \\ \bar y_+- y_+ \end{bmatrix}\right\|^2\\
	&\leq \left\|\begin{bmatrix} x-\bar x \\ y-\bar y \end{bmatrix}\right\|^2 + \left(\frac{\eta}{\eta-\rho}-2\right)\begin{bmatrix} \bar x- x \\ \bar y- y \end{bmatrix}^T\begin{bmatrix} \bar x_+- x_+ \\ \bar y_+- y_+ \end{bmatrix}.
	\end{align*}
	Notice that $\begin{bmatrix} \bar x- x \\ \bar y- y \end{bmatrix}^T\begin{bmatrix} \bar x_+- x_+ \\ \bar y_+- y_+ \end{bmatrix}$ is non-negative but the sign of $\left(\frac{\eta}{\eta-\rho}-2\right)$ may be positive or negative. If this coefficient is negative, we can upperbound the second term above by zero, giving $	\|\nabla L_\eta(z) - \nabla L_\eta(\bar z)\|^2 \leq \eta^2\left\|z-\bar z\right\|^2 .$
	If instead $\left(\frac{\eta}{\eta-\rho}-2\right)\geq 0$, then we have smoothness constant $|\eta^{-1}-\rho^{-1}|^{-1}$ as
	\begin{align*}
	\|\nabla L_\eta(z) - \nabla L_\eta(\bar z)\|^2 &\leq \eta^2\left(1 + \left(\frac{\eta}{\eta-\rho}-2\right)\frac{\eta}{\eta-\rho}\right)\left\|\begin{bmatrix} x-\bar x \\ y-\bar y \end{bmatrix}\right\|^2 \\
	&= \eta^2\left(\frac{\eta}{\eta-\rho}-1\right)^2\left\|\begin{bmatrix} x-\bar x \\ y-\bar y \end{bmatrix}\right\|^2 \\
	&= \left(\frac{\eta\rho}{\eta-\rho}\right)^2\left\|\begin{bmatrix} x-\bar x \\ y-\bar y \end{bmatrix}\right\|^2 \ .
	\end{align*}
	by Cauchy–Schwarz and using that $\left\|\begin{bmatrix} \bar x_+- x_+ \\ \bar y_+- y_+ \end{bmatrix}\right\| \leq \frac{\eta}{\eta-\rho}\left\|\begin{bmatrix} \bar x- x \\ \bar y- y \end{bmatrix}\right\|$.
\end{proof}
The setting of taking the Moreau envelope of a convex function gives a simpler smoothness bound of $\eta$ since having $\rho\leq 0$ implies $\eta = \max\{\eta, |\eta^{-1}-\rho^{-1}|^{-1}\}$. The same simplification holds when applying our saddle envelope machinery to convex-concave problems: the saddle envelope of any convex-concave $L$ is $\eta$-smooth, matching the results of~\cite[Proposition 2.1]{Aze1988}.

\section{Interaction Dominant Regime}\label{sec:interaction-dominate}
Our theory for the saddle envelope $L_{\eta}(z)$ shows it is much more structured than the original objective function $L(z)$. 
Proposition~\ref{prop:hessians} established that for $x$ and $y$ interaction dominant problems, the saddle envelope is strongly convex-strongly concave. Proposition~\ref{prop:smoothness} established that the saddle envelope is always smooth (has a uniformly Lipschitz gradient). Both of these results hold despite us not assuming convexity, concavity, or smoothness of the original objective. Historically these two conditions are the key to linear convergence (see Theorem~\ref{thm:standard-contraction}) and indeed we find interaction dominance causes the proximal point method to linearly converge. The proof of this result is deferred to the end of the section.
\begin{theorem}\label{thm:two-sided-dominance}
	For any objective $L$ that is $\rho$-weakly convex-weakly concave and $\alpha > 0$-interaction dominant in both $x$ and $y$, the damped PPM~\eqref{eq:saddle-PPM} with $\eta$ and $\lambda$ satisfying
	$$ \lambda \leq 2\frac{\min\left\{1,(\eta/\rho-1)^2\right\}}{\eta/\alpha+1} $$
	linearly converges to the unique stationary point $(x^*,y^*)$ of~\eqref{eq:main-problem} with
	$$\left\|\begin{bmatrix} x_k-x^* \\ y_k-y^* \end{bmatrix}\right\|^2 \leq \left(1 - \frac{2\lambda}{\eta/\alpha+1} + \frac{\lambda^2}{\min\left\{1,(\eta/\rho-1)^2\right\}}\right)^{k}\left\|\begin{bmatrix} x_0-x^* \\ y_0-y^* \end{bmatrix}\right\|^2.$$
	For example, setting $\eta=2\rho$ and $\lambda=\frac{1}{1+\eta/\alpha} $, our convergence rate simplifies to
	$$\left\|\begin{bmatrix} x_k-x^* \\ y_k-y^* \end{bmatrix}\right\|^2 \leq \left(1-\frac{1}{(2\rho/\alpha+1)^2} \right)^{k}\left\|\begin{bmatrix} x_0-x^* \\ y_0-y^* \end{bmatrix}\right\|^2.$$
\end{theorem} 
\begin{remark}
	Theorem \ref{thm:two-sided-dominance} is valid even if $\alpha>0$-interaction dominance only holds locally. That is, as long as $\alpha$-interaction dominance holds within an $l_2$-ball around a local stationary point, and the initial point is sufficiently within this ball, then PPM converges linearly to this local stationary point. 
\end{remark}
\begin{remark}
	For $\mu$-strongly convex-strongly concave problems, this theorem recovers the standard proximal point convergence rate for any choice of $\eta>0$. In this case, we have $\rho=-\mu$, $\alpha=\mu$, and can set $\lambda=\frac{1}{\eta/\mu+1}$, giving a $O(\eta^2/\mu^2 \log(1/\varepsilon))$ convergence rate matching~\cite{rockafellar1976monotone}.
\end{remark}
\begin{remark}
	The $\alpha>0$-interaction dominance condition is tight for obtaining global linear convergence. A nonconvex-nonconcave quadratic example illustrating the sharpness of this boundary is presented in Section~\ref{subsec:tight}. Moreover, our example shows that it is sometimes necessary to utilize the damping parameter (that is, selecting $\lambda<1$) for PPM to converge.
\end{remark}

If we only have $\alpha>0$-interaction dominance with $y$, then the saddle envelope $L_\eta$ is still much more structured than the original objective $L$. In this case, $L_\eta(x,y)$ may still be nonconvex in $x$, but Proposition~\ref{prop:smoothness} ensures it is strongly concave in $y$. Then our theory allows us to extend existing convergence guarantees for nonconvex-concave problems to this larger class of $y$ interaction dominant problems. For example, Lin et al.~\cite{Jordan1906} recently showed that GDA with different, carefully chosen stepsize parameters for $x$ and $y$ will converge to a stationary point at a rate of $O(\varepsilon^{-2})$. We find that running the following damped proximal point method is equivalent to running their variant of GDA on the saddle envelope
\begin{equation}\label{eq:saddle-PPM-2variables}
\begin{bmatrix} x_{k+1} \\ y_{k+1} \end{bmatrix} = \begin{bmatrix} \lambda x_{k}^+ +(1-\lambda)x_k \\ \gamma y_{k}^+ +(1-\gamma)y_k \end{bmatrix} \text{ where } \begin{bmatrix} x_{k}^+ \\ y_{k}^+ \end{bmatrix} = \prox_\eta(x_k,y_k)
\end{equation}
for proper choice of the parameters $\lambda,\gamma \in [0,1]$.
From this, we derive the following sublinear convergence rate for nonconvex-nonconcave problems whenever $y$ interaction dominance holds, proven at the end of the section.
\begin{theorem}\label{thm:one-sided-dominance}
	For any objective $L$ that is $\rho$-weakly convex-weakly concave and $\alpha > 0$-interaction dominant in $y$, consider the PPM variant~\eqref{eq:saddle-PPM-2variables} with damping constants $\lambda=\Theta\left(\frac{\min\left\{1,|\eta/\rho-1|^3\right\}}{(1+\eta/\alpha)^2}\right)$ and $\gamma=\Theta\left(\min\left\{1,|\eta/\rho-1|\right\}\right)$. If the sequence $y_k$ is bounded\footnote{We do not believe this boundedness condition is fundamentally needed, but we make it to leverage the results of~\cite{Jordan1906} which utilize compactness.}, then a stationary point $\|\nabla L(x_{T}^+,y_{T}^+)\|\leq \varepsilon$ will be found by iteration $T\leq O\left(\varepsilon^{-2}\right)$.
\end{theorem}

\begin{remark}
	Symmetrically, we can guarantee sublinear convergence assuming only $x$-interaction dominance. Considering the problem of $\max_y\min_x L(x,y) = -\min_y\max_x-L(x,y)$, which is now interaction dominant with respect to the inner maximization variable, we can apply Theorem~\ref{thm:one-sided-dominance}. This reduction works since although the original minimax problem and this maximin problem need not have the same solutions, they always have the same stationary points.
\end{remark}	

\subsection{Proof of Theorem~\ref{thm:two-sided-dominance}}
Propositions~\ref{prop:hessians} and~\ref{prop:smoothness} show that $L_\eta$ is $\mu=(\eta^{-1}+\alpha^{-1})^{-1}$-strongly convex-strongly concave and has a $\beta=\max\{\eta, |\eta^{-1}-\rho^{-1}|^{-1}\}$-Lipschitz gradient. Having strong convexity and strong concavity ensures $L_\eta$ has a unique stationary point $(x^*,y^*)$, which in turn must be the unique stationary point of $L$ by Corollary~\ref{cor:stationarity}. 
Recall Corollary~\ref{cor:PPM-to-GDA} showed that the damped PPM~\eqref{eq:saddle-PPM} on $L$ is equivalent to GDA~\eqref{eq:GDA} with $s=\lambda/\eta$ on $L_\eta$. Then provided 
$$ \lambda \leq 2\frac{\min\left\{1,(\eta/\rho-1)^2\right\}}{\eta/\alpha+1} = \frac{2(\eta^{-1}+\alpha^{-1})^{-1}}{\max\{\eta^2, (\eta^{-1}-\rho^{-1})^{-2}\}}\ , $$
we have $s=\lambda/\eta\in(0,2\mu/\beta^2)$. Hence applying Theorem~\ref{thm:standard-contraction} shows the iterations of GDA (and consequently PPM) linearly converge to this unique stationary point as
\begin{align*}
\left\|\begin{bmatrix} x_k-x^* \\ y_k-y^* \end{bmatrix}\right\|^2 &\leq \left(1 - \frac{2\lambda}{\eta/\alpha+1} + \frac{\lambda^2}{\min\left\{1,(\eta/\rho-1)^2\right\}}\right)^{k}\left\|\begin{bmatrix} x_0-x^* \\ y_0-y^* \end{bmatrix}\right\|^2.
\end{align*}

\subsection{Proof of Theorem~\ref{thm:one-sided-dominance}}
Proposition~\ref{prop:hessians} shows that whenever interaction dominance holds for $y$ the saddle envelope is $\mu=(\eta^{-1}+\alpha^{-1})^{-1}$-strongly concave in $y$ and Proposition~\ref{prop:smoothness} ensures the saddle envelope has a $\beta=\max\{\eta, |\eta^{-1}-\rho^{-1}|^{-1}\}$-Lipschitz gradient. 
Recently, Lin et al.~\cite{Jordan1906} considered such nonconvex-strongly concave problems with a compact constraint $y\in D$. They analyzed the following variant of GDA
\begin{equation}\label{eq:GDA-2sizes}
\begin{bmatrix} x_{k+1} \\ y_{k+1} \end{bmatrix} = \mathrm{proj}_{\RR^n\times D}\left(\begin{bmatrix} x_{k} \\ y_{k} \end{bmatrix} + \begin{bmatrix} -\nabla_x L(x_k,y_k)/\eta_x \\ \nabla_y L(x_k,y_k)/\eta_y \end{bmatrix}\right)
\end{equation}
which projects onto the feasible region $\RR^n\times D$ each iteration and has different stepsize parameters $\eta_x$ and $\eta_y$ for $x$ and $y$. Lin et al.\ prove the following theorem showing a sublinear guarantee.
\begin{theorem}[Theorem 4.4 of~\cite{Jordan1906}]
	For any $\beta$-smooth, nonconvex-$\mu$-strongly concave $L$, let $\kappa=\beta/\mu$ be the condition number for $y$. Then for any $\varepsilon>0$, GDA with stepsizes $\eta^{-1}_x=\Theta(1/\kappa^2\beta)$ and $\eta^{-1}_y=\Theta(1/\beta)$ will find a point satisfying $\|\nabla L(x_T,y_T)\|\leq \varepsilon$ by iteration
	$$ T\leq O\left(\frac{\kappa^2\beta + \kappa \beta^2}{\varepsilon^2}\right).$$
\end{theorem}

Assuming that the sequence $y_k$ above stays bounded, this projected gradient method is equivalent to running GDA on our unconstrained problem by setting the domain of $y$ as a sufficiently large compact set to contain all the iterates. 
Consider setting the averaging parameters as $\lambda=\Theta(\eta/\kappa^2\beta)=\Theta\left(\frac{\min\left\{1,|\eta/\rho-1|^3\right\}}{(1+\eta/\alpha)^2}\right)$ and $\gamma=\Theta(\eta/\beta)=\Theta\left(\min\left\{1,|\eta/\rho-1|\right\}\right)$. Then using the gradient formula from Lemma~\ref{lem:grad-formula}, we see that the damped proximal point method~\eqref{eq:saddle-PPM-2variables} is equivalent to running GDA on the saddle envelope with $\eta_x=\eta/\lambda$ and $\eta_y=\eta/\gamma$:
\begin{align*}
\begin{bmatrix} x_{k+1} \\ y_{k+1} \end{bmatrix} = \begin{bmatrix} x_{k} \\ y_{k}\end{bmatrix} + \begin{bmatrix} -\nabla_xL_\eta(x_k,y_k)/\eta_x \\ \nabla_yL_\eta(x_k,y_k)/\eta_y\end{bmatrix}
&= \begin{bmatrix} \lambda x_{k}^+ +(1-\lambda)x_k \\ \gamma y_{k}^+ +(1-\gamma)y_k \end{bmatrix}. 
\end{align*}

Then the above theorem guarantees that running this variant of the proximal point method on $L$ (or equivalently, applying the GDA variant~\eqref{eq:GDA-2sizes} to the saddle envelope) will converge to a stationary point with $\|\nabla L_\eta(z_T)\| \leq \varepsilon$ within $T\leq O(\varepsilon^{-2})$ iterations. 
It immediate follows from the gradient formula that $z_T^+ = \prox_\eta(z_T)$ is approximately stationary for $L$ as $\|\nabla L(z_T^+)\|=\|\nabla L_\eta(z_T)\| \leq\varepsilon$.

\section{Interaction Weak Regime}\label{sec:interaction-weak}

Our previous theory showed that when the interaction between $x$ and $y$ is sufficiently strong, global linear convergence occurs. Now we consider when there is limited interaction between $x$ and $y$. At the extreme of having no interaction, nonconvex-nonconcave minimax optimization separates into nonconvex minimization and nonconcave maximization. On these separate problems, local convergence of the proximal point method is well-understood.
Here we show that under reasonable smoothness and initialization assumptions, this local convergence behavior extends to minimax problems with weak, but nonzero, interaction between $x$ and $y$.

To formalize this, we make the following regularity assumptions
\begin{align}
\|\nabla^2 L(z)\| \leq \beta &\ , \ \text{ for all } z\in\RR^n\times\RR^m \label{eq:bounded-hessian}\\ 
\|\nabla^2 L(z) - \nabla^2 L(\bar z)\| \leq H\|z-\bar z\| &\ , \ \text{ for all } z, \bar z\in\RR^n\times\RR^m \label{eq:lip-hessian}
\end{align}
and quantify how weak the interaction is by assuming
\begin{align}
\|\nabla_{xy}^2 L(z)\| \leq \delta \quad & , \ \text{ for all } z\in\RR^n\times\RR^m  \label{eq:interaction-bound} \\ 
\begin{cases} \|\nabla^2_{xx} L(x,y) - \nabla^2_{xx} L(x,\bar y)\| \leq \xi\|y-\bar y\| \\ \|\nabla^2_{yy} L(x,y) - \nabla^2_{yy} L(\bar x,y)\| \leq \xi\|x-\bar x\| \end{cases}& , \ \text{ for all } (x,y), (\bar x, \bar y)\in\RR^n\times\RR^m \  \label{eq:interaction-lip}
\end{align}
for some constants $\beta, H, \delta, \xi \geq 0$.
Here we are particularly interested in problems where $\delta$ and $\xi$ are sufficiently small. For example, the bilinear setting of~\eqref{eq:bilinear} satisfies this with $(\delta, \xi) = (\lambda_{max}(A), 0)$ and so we are considering small interaction matrices $A$.

For such problems, we consider an initialization for the proximal point method based on our motivating intuition that when there is no interaction, we can find local minimizers and maximizers with respect to $x$ and $y$. For a fixed point $z'=(x',y')$, we compute our PPM initialization $z_0=(x_0,y_0)$ as
\begin{align}\label{eq:initialization}
\begin{cases}
x_0 = \text{a local minimizer of } \min_u L(u,y')\ , \\
y_0 = \text{a local maximizer of } \max_v L(x',v)\ .
\end{cases}
\end{align}
These subproblems amount to smooth nonconvex minimization, which is well-studied (see for example \cite{lee2016gradient}), and so we take them as a blackbox.

The critical observation explaining why this is a good initialization is that provided $\delta$ and $\xi$ are small enough, we have (i) that the interaction dominance conditions~\eqref{eq:x-dominance} and~\eqref{eq:y-dominance} hold at $z_0$ with a nearly positive $\alpha=\alpha_0$, often with $\alpha_0>0$ and (ii) that $z_0$ is a nearly stationary point of $L$. Below we formalize each of these properties and arrive at conditions quantifying how small we need $\xi$ and $\delta$ to be for our local convergence theory to apply.
\begin{description}
	\item[(i)] First, we observe that the interaction dominance conditions~\eqref{eq:x-dominance} and~\eqref{eq:y-dominance} hold at $z_0$ with a nearly positive coefficient $\alpha_0$. Since $x_0$ and $y_0$ are local optimum, for some $\mu\geq 0$, we must have
	$$ \nabla^2_{xx} L(x_0,y') \succeq \mu I \ \ \ \text{and} \ \ \ -\nabla^2_{yy} L(x',y_0) \succeq \mu I \ . $$
	Then the Hessians at $z_0$ must be similarly bounded since the amount they can change is limited by~\eqref{eq:interaction-lip}. Hence
	$$ \nabla^2_{xx} L(z_0) \succeq (\mu-\xi\|y_0-y'\|)I \ \ \ \text{and} \ \ \ -\nabla^2_{yy} L(z_0)  \succeq (\mu-\xi\|x_0-x'\|)I \ .$$
	Adding a positive semidefinite term onto these (as is done in the definition of interaction dominance) can only increase the righthand-side above. In particular, we can bound the second term added in the interaction dominance conditions~\eqref{eq:x-dominance} and~\eqref{eq:y-dominance} as
	\begin{align*}
	\nabla^2_{xy} L(z_0)(\eta I -\nabla^2_{yy} L(z_0))^{-1}\nabla^2_{yx} L(z_0) &\succeq  \frac{\nabla^2_{xy} L(z_0)\nabla^2_{yx} L(z_0)}{\eta+\beta} \\
	& \succeq  \frac{\lambda_{min}(\nabla^2_{xy} L(z_0)\nabla^2_{yx} L(z_0))}{\eta+\beta} I \geq 0 ,\\
	\nabla^2_{yx} L(z_0)(\eta I +\nabla^2_{xx} L(z_0))^{-1}\nabla^2_{xy} L(z_0) &\succeq  \frac{\nabla^2_{yx} L(z_0)\nabla^2_{xy} L(z_0)}{\eta+\beta} \\
	& \succeq \frac{\lambda_{min}(\nabla^2_{yx} L(z_0)\nabla^2_{xy} L(z_0))}{\eta+\beta} I \geq 0 .
	\end{align*}	    
	Hence interaction dominance holds at $z_0$ in both $x$ and $y$ with 
	\begin{align*}
	\nabla^2_{xx} L(z_0)& + \nabla^2_{xy} L(z_0)(\eta I -\nabla^2_{yy} L(z_0))^{-1}\nabla^2_{yx} L(z_0)\\ &\ \ \succeq \left(\mu+ \frac{\lambda_{min}(\nabla^2_{xy} L(z_0)\nabla^2_{yx} L(z_0))}{\eta+\beta}-\xi\|y_0-y'\|\right) I \ , \\
	-\nabla^2_{yy} L(z_0)& + \nabla^2_{yx} L(z_0)(\eta I +\nabla^2_{xx} L(z_0))^{-1}\nabla^2_{xy} L(z_0)\\ &\ \ \succeq \left(\mu+ \frac{\lambda_{min}(\nabla^2_{yx} L(z_0)\nabla^2_{xy} L(z_0))}{\eta+\beta}-\xi\|x_0-x'\|\right) I \ .
	\end{align*}
	For our local linear convergence theory to apply, we need this to hold with positive coefficient. It suffices to have $\xi$ sufficiently small, satisfying
	\begin{equation} \label{eq:psd-bound}
	\begin{cases}
	\xi\|y_0-y'\| < \mu+\dfrac{\lambda_{min}(\nabla^2_{xy} L(z_0)\nabla^2_{yx} L(z_0))}{\eta+\beta} \\
	\xi\|x_0-x'\| < \mu+\dfrac{\lambda_{min}(\nabla^2_{yx} L(z_0)\nabla^2_{xy} L(z_0))}{\eta+\beta} \ .
	\end{cases}
	\end{equation}
	Note this is trivially the case for problems with bilinear interaction~\eqref{eq:bilinear} as $\xi=0$.
	It is also worth noting that even if $\mu=0$, the right-hand-sides above are still strictly positive if $\nabla_{xy} L(z_0)$ is full rank and the variable dimensions $n$ and $m$ of $x$ and $y$ are equal\footnote{This works since having full rank square $\nabla^2_{xy} L(z_0)$ implies that both of its squares $\nabla^2_{xy} L(z_0)\nabla^2_{yx} L(z_0)$ and $\nabla^2_{yx} L(z_0)\nabla^2_{xy} L(z_0)$ are full rank as well. Hence these squares must be strictly positive definite and as a result, have strictly positive minimum eigenvalues.}.
	
	\item[(ii)] Next, we observe that $z_0$ is nearly stationary by applying~\eqref{eq:interaction-bound} and using the first-order optimality conditions of the subproblems~\eqref{eq:initialization}:
	\begin{align*}
	\|\nabla L(z_0)\| \leq \left\|\begin{bmatrix} \nabla_{x} L(x_0, y') \\ \nabla_{y} L(x',y_0) \end{bmatrix} \right\| + \delta\|z_0-z'\| = \delta\|z_0-z'\|.  
	\end{align*}
	For our convergence theory, this gradient needs to be sufficiently small
	\begin{equation} \label{eq:local-bound}
	\delta\|z_0-z'\| \leq \frac{\alpha_0(\eta-\rho)}{ 2\left(1+\frac{4\sqrt{2}(\eta + \alpha_0/2)}{\alpha_0}+\frac{4\sqrt{2}\beta(\eta + \alpha_0/2)}{\alpha_0(\eta-\rho)}\right)H\left(1+\frac{2\delta}{\eta-\rho}+\frac{\delta^2}{(\eta-\rho)^2}\right)}.
	\end{equation}
\end{description}

Under these conditions, we have the following linear convergence guarantee.
\begin{theorem}\label{thm:interaction-weak-convergence}
	For any objective $L$ satisfying weak convexity-concavity~\eqref{eq:weak}, the smoothness conditions~\eqref{eq:bounded-hessian} and \eqref{eq:lip-hessian}, and the interaction bounds~\eqref{eq:interaction-bound} and~\eqref{eq:interaction-lip}, consider the damped PPM~\eqref{eq:saddle-PPM}  with initialization $(x_0,y_0)$ given by~\eqref{eq:initialization} and $\eta$ and $\lambda$ satisfying
	$$ \lambda \leq 2\frac{\min\left\{1,(\eta/\rho-1)^2\right\}}{2\eta/\alpha_0+1} \ .$$
	Then PPM linearly converges to a nearby stationary point $(x^*,y^*)$ of~\eqref{eq:main-problem} with 
	$$\left\|\begin{bmatrix} x_k-x^* \\ y_k-y^* \end{bmatrix}\right\|^2 \leq \left(1 - \frac{2\lambda}{2\eta/\alpha_0+1} + \frac{\lambda^2}{\min\left\{1,(\eta/\rho-1)^2\right\}}\right)^{k}\left\|\begin{bmatrix} x_0-x^* \\ y_0-y^* \end{bmatrix}\right\|^2$$
	provided $\delta$ and $\xi$ are small enough to satisfy~\eqref{eq:psd-bound} and~\eqref{eq:local-bound}.
\end{theorem}

\subsection{Proof of Theorem~\ref{thm:interaction-weak-convergence}}
Our proof of this local convergence guarantee considers two sets centered at $(x_0,y_0)$: An inner region $\Binner=B(x_0,r)\times B(y_0,r)$ with radius
$$ r := \frac{4(\eta + \alpha_0/2)}{\alpha_0}\frac{\|\nabla L(z_0)\|}{\eta-\rho}$$
and an outer ball $\Bouter=B((x_0,y_0), R)$ with radius 
$$ R := \left(1+\frac{4\sqrt{2}(\eta + \alpha_0/2)}{\alpha_0}+\frac{4\sqrt{2}\beta(\eta + \alpha_0/2)}{\alpha_0(\eta-\rho)}\right)\frac{\|\nabla L(z_0)\|}{\eta-\rho} \geq \sqrt{2}r \ .$$

Thus $\Binner\subseteq\Bouter$. The following lemma shows that the $\alpha_0> 0$-interaction dominance at $z_0$ (following from our initialization procedure) extends to give $\alpha_0/2$-interaction dominance on the whole outer ball $\Bouter$.
\begin{lemma}
	On $\Bouter$, $\alpha_0/2$-iteration dominance holds in both $x$ and $y$.
\end{lemma}
\begin{proof}
	First, observe that the functions defining the interaction dominance conditions~\eqref{eq:x-dominance} and~\eqref{eq:y-dominance}
	\begin{align*}
	\nabla^2_{xx} L(z) + \nabla^2_{xy} L(z)(\eta I -\nabla^2_{yy} L(z))^{-1}\nabla^2_{yx} L(z), \\
	-\nabla^2_{yy} L(z) + \nabla^2_{yx} L(z)(\eta I +\nabla^2_{xx} L(z))^{-1}\nabla^2_{xy} L(z) 
	\end{align*}
	are both uniformly Lipschitz with constant\footnote{This constant follows from multiple applications of the ``product rule''-style formula that $A(z)B(z)$ is uniformly $(a'b + ab')$-Lipschitz provided $A(z)$ is bounded by $a$ and $a'$-Lipschitz and $B(z)$ is bounded by $b$ and $b'$-Lipschitz: any two points $z,z'$ have
		\begin{align*}
		\|A(z)B(z) - A(z')B(z')\| 
		& \leq \|A(z)B(z) - A(z')B(z) \| + \|A(z')B(z)- A(z')B(z')\|\\
		&\leq (a'b + b'a)\|z-z'\| .
		\end{align*}}
	$$ H\left(1+\frac{2\delta}{\eta-\rho}+\frac{\delta^2}{(\eta-\rho)^2}\right). $$
	Then our Lipschitz constant follows by observing the component functions defining it satisfy the following: $\nabla^2_{xx} L(z)$ and $\nabla^2_{yy} L(z)$ are $H$-Lipschitz, $\nabla^2_{xy} L(z)$ and its transpose $\nabla^2_{yx} L(z)$ are both $H$-Lipschitz and bounded in norm by $\delta$, and $(\eta I + \nabla^2_{xx} L(z))^{-1}$ and $(\eta I - \nabla^2_{yy} L(z))^{-1}$ are both $H/(\eta-\rho)^2$-Lipschitz and bounded in norm by $(\eta-\rho)^{-1}$.
	
	It follows that every $z\in\Bouter$ has $\alpha_0/2$-interaction dominance in $x$ as
	\begin{align*}
	& \nabla^2_{xx} L(z) + \nabla^2_{xy} L(z)(\eta I -\nabla^2_{yy} L(z))^{-1}\nabla^2_{yx} L(z)\\
	&\succeq \nabla^2_{xx} L(z_0) + \nabla^2_{xy} L(z_0)(\eta I -\nabla^2_{yy} L(z_0))^{-1}\nabla^2_{yx} L(z_0) - H\left(1+\frac{2\delta}{\eta-\rho}+\frac{\delta^2}{(\eta-\rho)^2}\right)RI\\
	&\succeq \nabla^2_{xx} L(z_0) + \nabla^2_{xy} L(z_0)(\eta I -\nabla^2_{yy} L(z_0))^{-1}\nabla^2_{yx} L(z_0) - \alpha_0/2 I\\
	&\succeq \alpha_0 I - \alpha_0/2 I = \alpha_0/2 I
	\end{align*}
	where the first inequality uses Lipschitz continuity, the second inequality uses our assumed condition~\eqref{eq:local-bound} of $H \left(1+\frac{2\delta}{\eta-\rho}+\frac{\delta^2}{(\eta-\rho)^2}\right) R\leq \alpha_0/2$, and the third inequality uses the $\alpha_0$-interaction dominance at $z_0$. Symmetric reasoning shows $\alpha_0/2$-interaction dominance in $y$ holds for each $z\in\Bouter$ as well
\end{proof}

From this, interaction dominance on the outer ball suffices to ensure the saddle envelope is strongly convex-strongly concave on the inner square.
\begin{lemma}
	The saddle envelope is $(\eta^{-1}+(\alpha_0/2)^{-1})^{-1}$-strongly convex-strongly concave on $\Binner$.
\end{lemma}
\begin{proof}
	Given $\alpha_0/2$-interaction dominance holds on $\Bouter$, it suffices to show that for any $z=(x,y)\in \Binner$, the proximal step $z_+ = \prox_\eta(z) \in \Bouter$ as we can then apply the Hessian bounds from Proposition \ref{prop:hessians} to show strong convexity and strong concavity.
	
	Define the function underlying the computation of the proximal step at $(x,y)$ as
	$$ M(u,v) = L(u,v) + \frac{\eta}{2}\|u-x\|^2 -\frac{\eta}{2}\|v-y\|^2. $$
	Our choice of $\eta>\rho$ ensures that $M$ is $(\eta-\rho)$-strongly convex-strongly concave.
	Thus applying Lemma~\ref{lem:helper0} and then the $\beta$-Lipschitz continuity of $\nabla L(z)$ implies
	$$ \left\|\begin{bmatrix} x-x_+ \\ y-y_+\end{bmatrix}\right\| \leq \frac{\|\nabla M(x,y)\|}{\eta-\rho} = \frac{\|\nabla L(x,y)\|}{\eta-\rho} \leq \frac{\|\nabla L(x_0,y_0)\|+\beta\sqrt{2}r}{\eta-\rho}\ . $$
	Hence $\|z_0-z_+\|\leq \|z_0-z\|+\|z-z_+\|\leq \sqrt{2}r+\frac{\|\nabla L(z_0)\|+\beta\sqrt{2}r}{\eta-\rho}=R$.
\end{proof}

Armed with the knowledge that interaction dominance holds on $\Binner$, we return to the proof of Theorem \ref{thm:interaction-weak-convergence}. Observe that the gradient of the saddle envelope at $z_0=(x_0,y_0)$ is bounded by Lemma~\ref{lem:grad-formula} and Lemma~\ref{lem:helper0} as
$$\|\nabla L_\eta(z_0)\| = \|\eta(z_0-z^+_0)\| \leq \frac{\eta}{\eta-\rho}\|\nabla M_0(z_0)\|=\frac{\eta}{\eta-\rho}\|\nabla L(z_0)\| $$
where $z^+_0=\prox_\eta(z_0)$ and $M_0(u,v)= L(u,v)+ \frac{\eta}{2}\|u-x_0\|^2 -\frac{\eta}{2}\|v-y_0\|^2$ is the $\eta-\rho$-strongly convex-strongly concave function defining it.
Now we have shown all of the conditions necessary to apply Theorem~\ref{thm:standard-contraction} on the square $B(x_0, r)\times B(y_0,r)$ with
$$ r = \frac{4(\eta + \alpha_0/2)\|\nabla L(z_0)\|}{\alpha_0(\eta-\rho)} =\frac{2\|\nabla L_\eta(z_0)\|}{\mu} $$
upon which the saddle envelope is $\mu=(\eta^{-1}+(\alpha_0/2)^{-1})^{-1}$-strongly convex-strongly concave and $\beta=\max\{\eta, |\eta^{-1}-\rho^{-1}|^{-1}\}$-smooth.
Hence applying GDA with $s=\lambda/\eta$ to the saddle envelope produces iterates $(x_k,y_k)$ converging to a stationary point $(x^*,y^*)$ with
$$\left\|\begin{bmatrix} x_k-x^* \\ y_k-y^* \end{bmatrix}\right\|^2 \leq \left(1 - \frac{2\lambda}{\eta(\eta^{-1}+(\alpha_0/2)^{-1})} + \frac{\lambda^2}{\eta^2(\eta^{-1}-\rho^{-1})^2}\right)^{k}\left\|\begin{bmatrix} x_0-x^* \\ y_0-y^* \end{bmatrix}\right\|^2 . $$
By Corollary~\ref{cor:stationarity}, $(x^*,y^*)$ must also be a stationary point of $L$. Further, by Corollary~\ref{cor:PPM-to-GDA}, this sequence $(x_k,y_k)$ is the same as the sequence generated by running the damped PPM on~\eqref{eq:main-problem}.

\section{Interaction Moderate Regime}\label{sec:interaction-moderate}
Between the interaction dominant and interaction weak regimes, the proximal point method may diverge or cycle indefinitely (recall our introductory example in Figure~\ref{fig:sample-path} where convergence fails in this middle regime). 
We begin by considering the behavior of the proximal point method when applied to a nonconvex-nonconcave quadratic example. From this, our interaction dominance condition is tight, exactly describing when our example converges.

\subsection{Divergence and Tightness of the Interaction Dominance Regime} \label{subsec:tight}
Consider the following nonconvex-nonconcave quadratic minimax problem of
\begin{equation} \label{eq:quadratic}
\min_{x\in\RR^n}\max_{y\in\RR^n} L(x,y)=\frac{-\rho}{2}\|x\|^2 + ax^Ty - \frac{-\rho}{2}\|y\|^2
\end{equation}
where $a\in \RR$ controls the size of the interaction between $x$ and $y$ and $\rho\geq 0$ controls how weakly convex-weakly concave the problem is. Notice this problem has a stationary point at the origin. Even though this problem is nonconvex-nonconcave, PPM will still converge to the origin for some selections of $a$, $\rho$, and $\eta$. Examining our interaction dominance conditions~\eqref{eq:x-dominance} and~\eqref{eq:y-dominance}, this example is $\alpha=-\rho+a^2/(\eta-\rho)$-interaction dominant in both $x$ and $y$.

For quadratic problems, PPM always corresponds to the matrix multiplication. In the case of~\eqref{eq:quadratic}, the damped PPM iteration is given by
\begin{align*}
\begin{bmatrix} x_{k+1} \\ y_{k+1} \end{bmatrix} &= (1-\lambda)\begin{bmatrix} x_{k} \\ y_{k} \end{bmatrix} + \lambda\begin{bmatrix} (1-\rho/\eta)I & aI/\eta \\ -aI/\eta & (1-\rho/\eta)I\end{bmatrix}^{-1}\begin{bmatrix} x_{k} \\ y_{k} \end{bmatrix} \\ 
&= (1-\lambda)\begin{bmatrix} x_{k} \\ y_{k} \end{bmatrix} + \frac{\lambda\eta}{\eta-\rho}\left(\begin{bmatrix} I & aI/(\eta-\rho) \\ -aI/(\eta-\rho) & I\end{bmatrix}\right)^{-1}\begin{bmatrix} x_{k} \\ y_{k} \end{bmatrix} \\
& = (1-\lambda)\begin{bmatrix} x_{k} \\ y_{k} \end{bmatrix} + \frac{\lambda\eta}{a^2/(\eta-\rho)+\eta-\rho}\begin{bmatrix} I & -aI/(\eta-\rho) \\ aI/(\eta-\rho) & I\end{bmatrix}\begin{bmatrix} x_{k} \\ y_{k} \end{bmatrix} \\
& = \begin{bmatrix} CI & -DI \\ DI & CI\end{bmatrix}\begin{bmatrix} x_{k} \\ y_{k} \end{bmatrix}
\end{align*}
for constants
$ C  = 1 - \dfrac{\lambda\alpha}{\eta+\alpha}$ 
and $ D = \dfrac{\lambda\eta a}{(\eta+\alpha)(\eta-\rho)}$. 
Notice that these constants are well-defined since $\eta-\rho>0$ and  $\eta+\alpha>0$ (even if $\alpha$ is negative) since $\eta>\rho$ and $\alpha \geq -\rho$.
Matrix multiplication of this special final form has the following nice property for any $z$,
\begin{equation}
\left\|\begin{bmatrix} CI & -DI \\ DI & CI\end{bmatrix}z\right\|^2 = (C^2+D^2)\|z\|^2. 
\end{equation}
Hence this iteration will globally converge to the origin exactly when
\begin{align*}
\left(1 - \frac{\lambda\alpha}{\eta+\alpha}\right)^2+\left(\frac{\lambda\eta a}{(\eta+\alpha)(\eta-\rho)}\right)^2 < 1 \ .
\end{align*}
Likewise, the damped proximal point method will cycle indefinitely when this holds with equality and diverges when it is strictly violated.
As a result, violating $\alpha> 0$-interaction dominance (that is, having $\alpha \leq 0$) leads to divergence in~\eqref{eq:quadratic} for any choice of the averaging parameter $\lambda\in(0,1]$ since this forces $C\geq 1$ (and so $C^2+D^2>1$). Hence our interaction dominance boundary is tight.

Further, this example shows that considering the damped proximal point method (as opposed to fixing $\lambda=1$) is necessary to fully capture the convergence for interaction dominant problems. For example, setting $\rho=1, a=2, \eta=3$ has $\alpha=1$-interaction dominance in $x$ and $y$ and converges exactly when
$$ (1-\lambda/4)^2 + (3\lambda/4)^2 < 1 $$
which is satisfied when $\lambda\in(0,0.8)$, but not by the undamped proximal point method with $\lambda=1$. Our theory from Theorem~\ref{thm:two-sided-dominance} is slightly more conservative, guaranteeing convergence whenever
$ \lambda \leq 0.5 = 2\min\left\{1,(\eta/\rho-1)^2\right\}/(\eta/\alpha+1)$. 

\subsection{A Candidate Lyapunov for Interaction Moderate Problems}
The standard analysis of gradient descent on nonconvex optimization relies on the fact that the function value monotonically decays every iteration. However, such properties fail to hold in the nonconvex-nonconcave minimax setting: the objective is neither monotonically decreasing nor increasing while PPM runs. Worse yet, since we know the proximal point method may cycle indefinitely with gradients bounded away from zero (for example, recall the interaction moderate regime trajectories in Figure~\ref{fig:sample-path}), no ``Lyapunov''-type quantity can monotonically decrease along the iterates of the proximal point method.

In order to obtain a similar analysis as the standard nonconvex optimization approach, we propose to study the following ``Lyapunov'' function, which captures the difference between smoothing over $y$ and smoothing over $x$ using the classic Moreau envelope,
\begin{align}
\mathcal{L}(x,y) := & -e_\eta\{-L(x,\cdot)\}(y) - e_\eta\{L(\cdot,y)\}(x) \ . \label{eq:lyapunov} 
\end{align}
The following proposition establishes structural properties supporting our consideration of $\mathcal{L}(x,y)$.

\begin{theorem}\label{thm:lyapunov-properties}
	The Lyapnuov $\mathcal{L}(x,y)$ has the following structural properties:
	\begin{enumerate}
		\item $\mathcal{L}(x,y) \geq 0$,
		\item When $\eta>\rho$, $\mathcal{L}(x,y) = 0$ if and only if $(x,y)$ is a stationary point to $L(x,y)$,
		\item When $\eta=0$, $\mathcal{L}(x,y)$ recovers the well-known primal-dual gap of $L(x,y)$
		$$ \mathcal{L}(x,y) = \max_{v} L(x,v) - \min_{u}L(u,y).$$
	\end{enumerate}
\end{theorem}
\begin{proof}
	Recall that a Moreau envelope $e_\eta\{f(\cdot)\}(x)$ provides a lower bound~\eqref{eq:moreau-lower-bound} on $f$ everywhere. Hence $e_\eta\{-L(x,\cdot)\}(y) \leq -L(x,y)$ and $e_\eta\{L(\cdot,y)\}(x) \leq L(x,y)$, and so our proposed Lyapunov is always nonnegative since
	$$ \mathcal{L}(x,y) = -e_\eta\{-L(x,\cdot)\}(y) - e_\eta\{L(\cdot,y)\}(x) \geq L(x,y)-L(x,y)=0 \ .$$
	Further, it follows from~\eqref{eq:moreau-grad} that for any $\rho$-weakly convex function $f$, selecting $\eta>\rho$ ensures the Moreau envelope equals the given function precisely at its stationary point. Then the preceding nonnegativity argument holds with equality if and only if 
	$$ \nabla_y -L(x,\cdot)(y) = 0 \ \ \ \text{and} \ \ \ \nabla_x L(\cdot,y)(x) = 0 \ . $$
	Hence we have $\mathcal{L}(x,y)=0 \iff \nabla L(x,y)=0$. 
	Lastly, when $\eta=0$, we have
	\begin{align*}
	\mathcal{L}(x,y) &= -\min_{v}\left\{-L(x,v) + \frac{\eta}{2}\|v-y\|^2\right\} - \min_{u}\left\{-L(u,y) + \frac{\eta}{2}\|u-x\|^2\right\}\\
	&= \max_{v} L(x,v) - \min_{u}L(u,y) \ ,
	\end{align*}
	recovering the primal-dual gap for $L(x,y)$. 
\end{proof}
For example, computing the Moreau envelopes defining $\mathcal{L}(z)$ for~\eqref{eq:quadratic} gives
\begin{align}
e_\eta\{L(\cdot, y)\}(x) &= \frac{1}{2}(\eta^{-1}-\rho^{-1})^{-1}\|x\|^2 + \frac{\eta a}{\eta-\rho}x^Ty - \frac{\alpha}{2}\|y\|^2 \label{eq:x-smoothing}\\
e_\eta\{-L(x,\cdot)\}(y) &= -\frac{\alpha}{2}\|x\|^2 - \frac{\eta a}{\eta-\rho}x^Ty + \frac{1}{2}(\eta^{-1}-\rho^{-1})^{-1}\|y\|^2\label{eq:y-smoothing}
\end{align}
where $\alpha = -\rho+a^2/(\eta-\rho)$ is this problem's interaction dominance. Hence 
\begin{align*}
\mathcal{L}(z) 
&=\frac{1}{2}\left(\alpha - (\eta^{-1}-\rho^{-1})^{-1} \right)\|z\|^2 \ .
\end{align*}
Noting that $\alpha \geq -\rho$ and $ -(\eta^{-1}-\rho^{-1})^{-1} > -\rho$, we see that the origin is the unique minimizer of $\mathcal{L}(z)$ and consequently the unique stationary point of $L$. In this case, minimizing $\mathcal{L}(z)$ is simple convex optimization.

Future works could identify further tractable nonconvex-nonconcave problem settings where algorithms can minimize $\mathcal{L}(x,y)$ instead as all of its global minimums are stationary points of the original objective. Since this problem is purely one of minimization, cycling can be ruled out directly.
As previously observed, the proximal point method is not such an algorithm since it may fall into a cycle and fail to monotonically decrease $\mathcal{L}(z)$. Instead, we find the following weakened descent condition for $\mathcal{L}(z)$, relating its change to our $\alpha$-interaction dominance conditions. Note that this result holds regardless of whether the interaction dominance parameter $\alpha$ is positive or negative.
\begin{theorem}\label{thm:middle-case-upper-bound}
	For any $\rho$-weakly convex-weakly concave, $\alpha\in\RR$-interaction dominant in $x$ and $y$ problem, any $z\in\RR^n\times\RR^m$ has $z_+=\prox_\eta(z)$ satisfy
	$$ \mathcal{L}(z_+) \leq \mathcal{L}(z) - \frac{1}{2}\left(\alpha+(\eta^{-1}-\rho^{-1})^{-1}\right)\|z_+-z\|^2 \ .$$
\end{theorem}
\begin{remark}
	This upper bound is attained by our example diverging problem~\eqref{eq:quadratic}. This is example attains our bound since the proof of Theorem~\ref{thm:middle-case-upper-bound} only introduces inequalities by using the following four Hessian bounds for every $(u,v)$
	\begin{align*}
	\nabla^2_{xx}-e_\eta\{-L(u,\cdot)\}(v) \succeq \alpha I\ , & \ \  \nabla^2_{yy}-e_\eta\{-L(u,\cdot)\}(v) \preceq -(\eta^{-1}-\rho^{-1})^{-1} I\ , \\
	\nabla^2_{yy}-e_\eta\{L(\cdot,v)\}(u) \succeq \alpha I\ , & \ \  \nabla^2_{xx}-e_\eta\{L(\cdot,v)\}(u) \preceq -(\eta^{-1}-\rho^{-1})^{-1} I \ .
	\end{align*}
	Observing that all four of these bounds hold with equality everywhere in \eqref{eq:x-smoothing} and \eqref{eq:y-smoothing} shows our recurrence holds with equality.
\end{remark}
\begin{remark}
	For generic minimax problems, Theorem~\ref{thm:middle-case-upper-bound} bounds how quickly PPM can diverge. For any objective $L$ that is $l$-Lipschitz and nearly convex-concave, satisfying weak convexity-weak concavity~\eqref{eq:weak} with some $\rho=\epsilon$. Then since $\alpha\geq-\rho=-\epsilon$, the Lyapanov increases by at most $O(\epsilon)$ as
	$$ \mathcal{L}(z_+) -\mathcal{L}(z) \leq  - \frac{1}{2}\left(\alpha-\frac{\eta\rho}{\eta-\rho}\right)\|\nabla L(z_+)/\eta\|^2 \leq \frac{\epsilon l^2}{2\eta^2}\left(1+\frac{\eta}{\eta-\epsilon}\right) \approx \frac{\epsilon l^2}{\eta^2} \ .$$
\end{remark}

\subsection{Proof of Theorem~\ref{thm:middle-case-upper-bound}}\label{sec:proof-hard}
First, we bound the Hessians of the functions defining our Lyapunov $\mathcal{L}(z)$.
\begin{lemma}\label{lem:one-side-smoothing}
	If the $x$-interaction dominance~\eqref{eq:x-dominance} holds with $\alpha\in\RR$, the function $e_\eta\{L(\cdot,y)\}(x)$ has Hessians in $x$ and $y$ bounded by
	$$ (\eta^{-1}-\rho^{-1})^{-1} I \preceq \nabla^2_{xx} e_\eta\{L(\cdot,y)\}(x) \preceq \eta I \ \ \text{and} \ \  \nabla^2_{yy} e_\eta\{L(\cdot,y)\}(x) \preceq -\alpha I \ . $$
	Symmetrically, if the $y$-interaction dominance~\eqref{eq:y-dominance} holds with $\alpha\in\RR$, 
	$$ \nabla^2_{xx} e_\eta\{-L(x,\cdot)\}(y) \preceq -\alpha I  \ \ \text{and} \ \ (\eta^{-1}-\rho^{-1})^{-1} I \preceq \nabla^2_{yy} e_\eta\{-L(x,\cdot)\}(y) \preceq \eta I.$$
\end{lemma}
\begin{proof}
	For the Hessian bound in the $x$ variable, this follows directly from the Moreau envelope Hessian bounds~\eqref{eq:moreau-hessian-bounds}.
	Considering $e_\eta\{L(\cdot,y)\}(x)$ as a function of $y$, we find that its gradient is given by $\nabla_{y} e_\eta\{L(\cdot,y)\}(x) = \nabla_{y} L(x_+,y)$ and Hessian $\nabla^2_{yy} e_\eta\{L(\cdot,y)\}(x)$ is given by
	\begin{align*}
	\nabla^2_{yy} L(x_+,y) - \nabla^2_{yx} L(x_+,y)(\eta I  +\nabla^2_{xx} L(x_+,y))^{-1}\nabla^2_{xy} L(x_+,y)
	\end{align*}
	where $x_+ = \mathrm{argmin}_u L(u,y) +\frac{\eta}{2}\|u-x\|^2$. Noting that this Hessian matches the $\alpha$-interaction dominance condition~\eqref{eq:y-dominance} gives our bound on $-\nabla^2_{yy} e_\eta\{L(\cdot,y)\}(x)$.
	
	All that remains is to derive our claimed gradient and Hessian formulas in $y$. Consider a nearby point $y^\Delta=y+\Delta$ and denote $x^\Delta_+ = \mathrm{argmin}_u L(u,y^\Delta) +\frac{\eta}{2}\|u-x\|^2$. 
	Consider the second-order Taylor model of the objective $L$ around $(x_+,y)$ denoted by $\widetilde L(u,v)$ with value
	\begin{align*}
	&L(x_+,y) + \begin{bmatrix}\nabla_x L(x_+,y) \\ \nabla_y L(x_+,y)\end{bmatrix}^T\begin{bmatrix} u-x_+ \\ v-y \end{bmatrix} 
	+ \frac{1}{2}\begin{bmatrix} u-x_+ \\ v-y \end{bmatrix}^T\begin{bmatrix} \nabla^2_{xx} L(x_+,y) & \nabla^2_{xy} L(x_+,y) \\ \nabla^2_{yx} L(x_+,y) & \nabla^2_{yy} L(x_+,y) \end{bmatrix}\begin{bmatrix} u-x_+ \\ v-y \end{bmatrix} \ .
	\end{align*}
	
	Denote the $\widetilde x^\Delta_+ = \mathrm{argmin}_u \widetilde L(u,y^\Delta) +\frac{\eta}{2}\|u-x\|^2$. Noting this point is uniquely defined by its first-order optimality conditions, we have
	\begin{align*}
	&\nabla_x L(x_+,y) + \nabla_{xx}^2L(x_+,y)(\widetilde x^\Delta_+ - x_+) + \nabla_{xy}^2L(x_+,y)\Delta +\eta(\widetilde x^\Delta_+ - x)= 0 \ ,\\
	&\implies   (\eta I + \nabla^2_{xx} L(x_+,y))(\widetilde x^\Delta_+-x_+) = -\nabla_{xy}^2L(x_+,y)\Delta \ ,\\
	&\implies  \widetilde x^\Delta_+-x_+ = -(\eta I + \nabla^2_{xx} L(x_+,y))^{-1}\nabla_{xy}^2L(x_+,y)\Delta \ .
	\end{align*}
	Denote the proximal subproblem objective by $ M^\Delta (u,v) = L(u,y^\Delta) + \frac{\eta}{2}\|u-x\|^2$ and its approximation by $\widetilde M^\Delta (u,v) = \widetilde L(u,y^\Delta) + \frac{\eta}{2}\|u-x\|^2$.
	Noting that $\|\nabla_x \widetilde M^\Delta (x_+,y^\Delta)\| = \|\nabla^2_{xy} L(x_+, y)\Delta\|$, the $(\eta-\rho)$-strongly convexity of $\widetilde M^\Delta$ bounds the distance to its minimizer by
	$$ \|x_+-\widetilde x_+^\Delta\|\leq \frac{\|\nabla^2_{xy}  L(x_+, y)\Delta\|}{\eta-\rho} = O(\|\Delta\|)\ .$$
	Consequently, we can bound difference in gradients between $L$ and its model $\widetilde L$ at $\widetilde x^\Delta_+$ by $ \|\nabla L(\widetilde x^\Delta_+, y^\Delta) - \nabla \widetilde L(\widetilde x^\Delta_+, y^\Delta)\| = o(\|\Delta\|)$. Therefore $\|\nabla M^\Delta(\widetilde x^\Delta_+,y^\Delta)\| = o(\|\Delta\|)$. Then using the strong convexity of $M^\Delta$ with this gradient bound, we conclude the distance from $\widetilde x^\Delta_+$ to the minimizer $x^\Delta_+$ is bounded by
	$ \|\widetilde x^\Delta_+ - x^\Delta_+\| = o(\|\Delta\|) . $
	Then our claimed gradient formula follows as
	\begin{align*}
	&e_\eta\{L(\cdot,y^\Delta)\}(x) - e_\eta\{L(\cdot,y)\}(x)\\
	&\qquad= L(x^\Delta_+,y^\Delta) +\frac{\eta}{2}\|x^\Delta_+ - x\|^2 - L(x_+,y) - \frac{\eta}{2}\|x_+ - x\|^2\\
	&\qquad= \begin{bmatrix} \nabla_x L(x_+,y) + \eta(x_+-x) \\ \nabla_y L(x_+,y) \end{bmatrix}^T\begin{bmatrix} x_+^\Delta-x_+ \\ \Delta \end{bmatrix} +o(\|\Delta\|)\\
	&\qquad= \nabla_y L(x_+,y)^T\Delta + o(\|\Delta\|)\ .
	\end{align*}
	Moreover, our claimed Hessian formula follows as
	\begin{align*}
	&\nabla_{y} e_\eta\{L(\cdot,y^\Delta)\}(x) - \nabla_{y} e_\eta\{L(\cdot,y)\}(x)\\
	&\qquad = \nabla_y L(x^\Delta_+,y^\Delta) - \nabla_y L(x_+,y) \\
	&\qquad  = \nabla_y \widetilde L(\widetilde x^\Delta_+,y^\Delta) - \nabla_y L(x_+,y) + o(\|\Delta\|)\\
	&\qquad  = \begin{bmatrix} \nabla^2_{xy} L(x_+,y) \\\nabla^2_{yy} L(x_+,y) \end{bmatrix}^T\begin{bmatrix} -(\eta I + \nabla^2_{xx} L(x_+,y))^{-1}\nabla_{xy}^2L(x_+,y)\Delta \\ \Delta \end{bmatrix}+ o(\|\Delta\|)\ .
	\end{align*}
\end{proof}

Notice that $-e_\eta\{-L(u,\cdot)\}(y)$ has gradient at $x_+$ of $\nabla_x -e_\eta\{-L(x_{+},\cdot)\}(y) = \nabla_x L(z_+)=\eta(x-x_+)$ and from Lemma~\ref{lem:one-side-smoothing} that its Hessian in $x$ is uniformly lower bounded by $\alpha I$. As a result, we have the following decrease in $-e_\eta\{-L(u,\cdot)\}(y)$ when moving from $x$ to $x_+$
\begin{align*}
-e_\eta\{-L(x_+,\cdot)\}(y) &\leq -e_\eta\{-L(x,\cdot)\}(y) + \nabla_x L(z_+)^T(x_+-x) -\frac{\alpha}{2}\|x_+-x\|^2\\
& = -e_\eta\{-L(x,\cdot)\}(y)-\left(\eta + \frac{\alpha}{2}\right)\|x_+-x\|^2 \ .
\end{align*}
From the gradient formula~\eqref{eq:moreau-grad}, we know that $\nabla_y-e_\eta\{-L(x_+,\cdot)\}(y) = \nabla_y L(z_+)=\eta(y_+-y)$ and from Lemma~\ref{lem:one-side-smoothing} that its Hessian in $y$ is uniformly bounded above by $-(\eta^{-1}-\rho^{-1})^{-1}I$. Then we can upper bound  the change in $-e_\eta\{-L(x_+,\cdot)\}(v)$ when moving from $y$ to $y_+$ as
\begin{align*}
&-e_\eta\{-L(x_+,\cdot)\}(y_+) + e_\eta\{-L(x_+,\cdot)\}(y)\\
&\qquad \leq \nabla_y L(z_+)^T(y_+-y) + \frac{-(\eta^{-1}-\rho^{-1})^{-1}}{2}\|y_+-y\|^2\\
&\qquad = \left(\eta -\frac{(\eta^{-1}-\rho^{-1})^{-1}}{2}\right)\|y_+-y\|^2 \ .
\end{align*}
Summing these two inequalities yields
\begin{align*}
&-e_\eta\{-L(x_+,\cdot)\}(y_+) + e_\eta\{-L(x,\cdot)\}(y)\\
&\qquad \leq \left(\eta -\frac{(\eta^{-1}-\rho^{-1})^{-1}}{2}\right)\|y_+-y\|^2 -\left(\eta + \frac{\alpha}{2}\right)\|x_+-x\|^2 \ .
\end{align*}
Symmetrically, the change in $-e_\eta\{L(\cdot,y)\}(x)$ from $z$ to $z_+$ is
\begin{align*}
&-e_\eta\{L(\cdot,y_+)\}(x_+) + e_\eta\{L(\cdot,y)\}(x)\\
&\qquad \leq \left(\eta -\frac{(\eta^{-1}-\rho^{-1})^{-1}}{2}\right)\|x_+-x\|^2 -\left(\eta + \frac{\alpha}{2}\right)\|y_+-y\|^2 \ .
\end{align*}
Summing these two symmetric results gives the claimed bound.

	\paragraph{Acknowledgements.}  Benjamin Grimmer was supported by the National Science Foundation Graduate Research Fellowship under Grant No. DGE-1650441.
	\medskip
	
	\small
	
	\bibliographystyle{plain}
	\bibliography{references}

	\appendix
	
\section{Sample Paths From Other First-Order Methods} \label{app:sample-paths}
\begin{figure}
	\begin{subfigure}[b]{0.245\textwidth}
		\includegraphics[width=\textwidth]{ppmA1.png}
		\caption{PPM w/~A=1}
	\end{subfigure}
	\begin{subfigure}[b]{0.245\textwidth}
		\includegraphics[width=\textwidth]{ppmA10.png}
		\caption{PPM w/~A=10}
	\end{subfigure}
	\begin{subfigure}[b]{0.245\textwidth}
		\includegraphics[width=\textwidth]{ppmA100.png}
		\caption{PPM w/~A=100}
	\end{subfigure}
	\begin{subfigure}[b]{0.245\textwidth}
		\includegraphics[width=\textwidth]{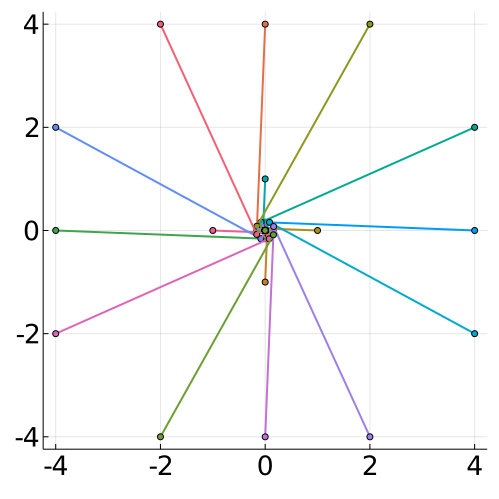}
		\caption{PPM w/~A=1000}
	\end{subfigure}
	\begin{subfigure}[b]{0.245\textwidth}
		\includegraphics[width=\textwidth]{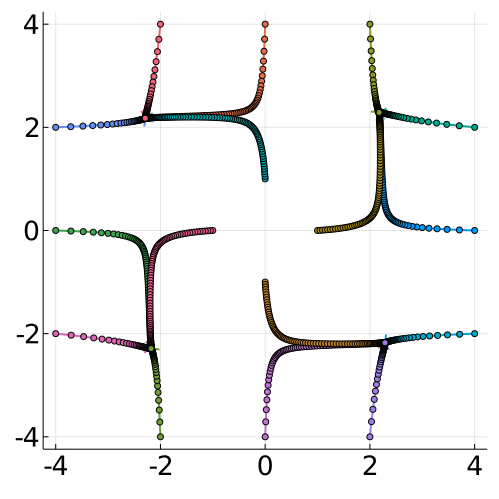}
		\caption{EGM w/~A=1}
	\end{subfigure}
	\begin{subfigure}[b]{0.245\textwidth}
		\includegraphics[width=\textwidth]{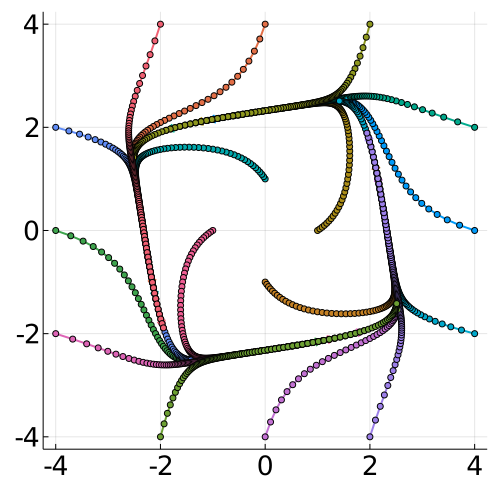}
		\caption{EGM w/~A=10}
	\end{subfigure}
	\begin{subfigure}[b]{0.245\textwidth}
		\includegraphics[width=\textwidth]{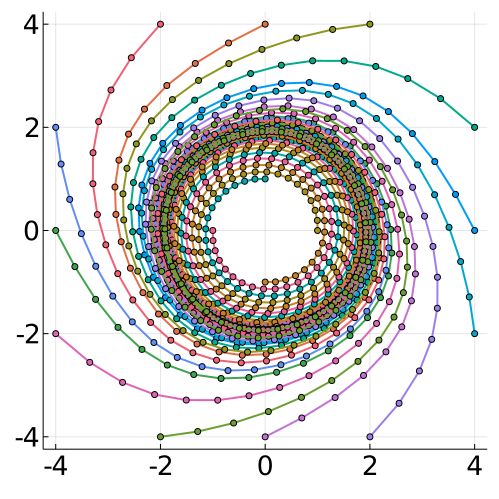}
		\caption{EGM w/~A=100}
	\end{subfigure}
	\begin{subfigure}[b]{0.245\textwidth}
		\includegraphics[width=\textwidth]{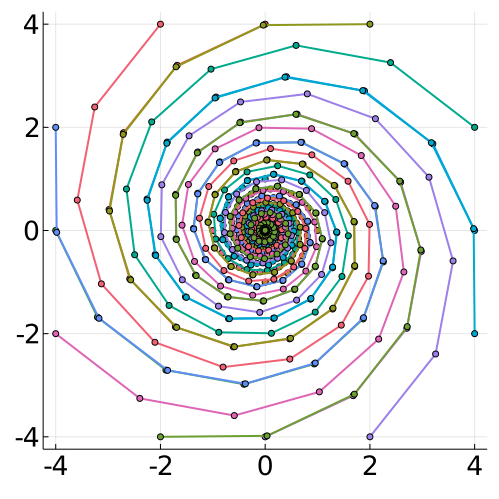}
		\caption{EGM w/~A=1000}
	\end{subfigure}
	\begin{subfigure}[b]{0.245\textwidth}
		\includegraphics[width=\textwidth]{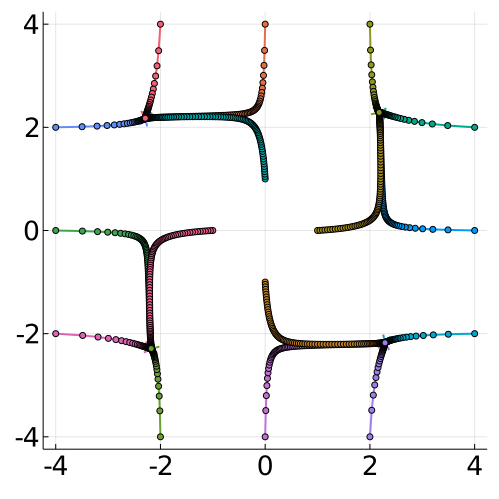}
		\caption{GDA w/~A=1}
	\end{subfigure}
	\begin{subfigure}[b]{0.245\textwidth}
		\includegraphics[width=\textwidth]{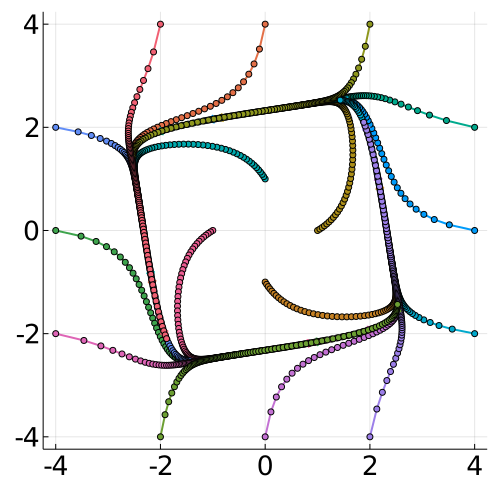}
		\caption{GDA w/~A=10}
	\end{subfigure}
	\begin{subfigure}[b]{0.245\textwidth}
		\includegraphics[width=\textwidth]{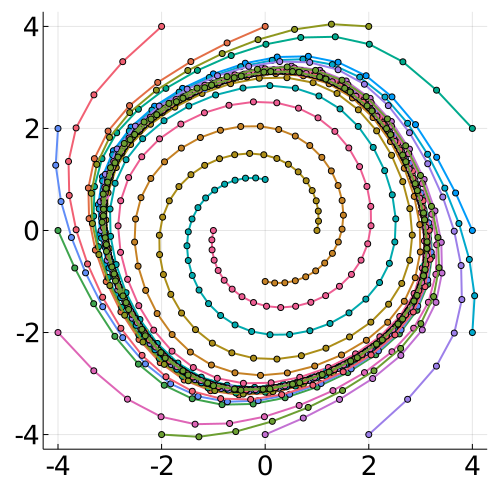}
		\caption{GDA w/~A=100}
	\end{subfigure}
	\begin{subfigure}[b]{0.245\textwidth}
		\includegraphics[width=\textwidth]{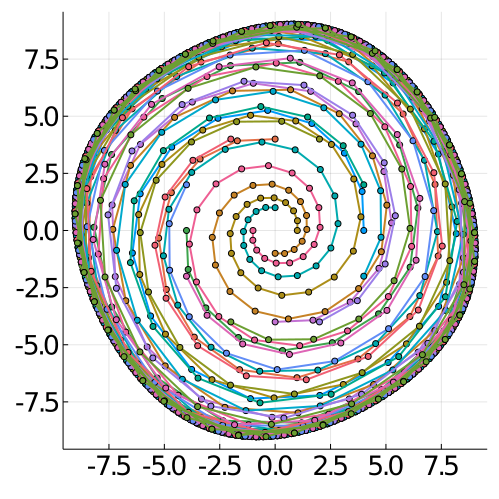}
		\caption{GDA w/~A=1000}
	\end{subfigure}
	\begin{subfigure}[b]{0.245\textwidth}
		\includegraphics[width=\textwidth]{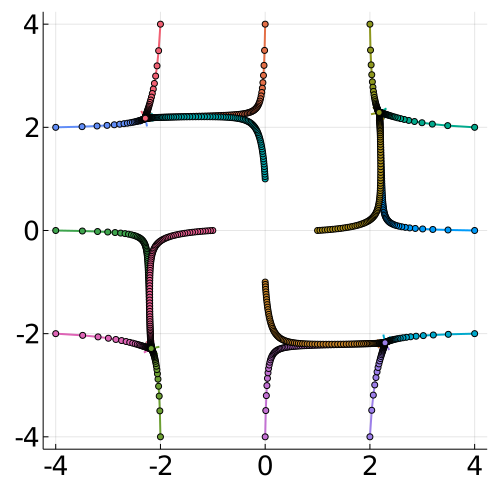}
		\caption{AGDA w/~A=1}
	\end{subfigure}
	\begin{subfigure}[b]{0.245\textwidth}
		\includegraphics[width=\textwidth]{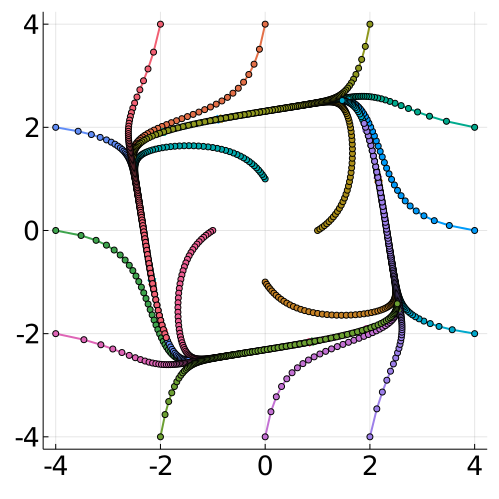}
		\caption{AGDA w/~A=10}
	\end{subfigure}
	\begin{subfigure}[b]{0.245\textwidth}
		\includegraphics[width=\textwidth]{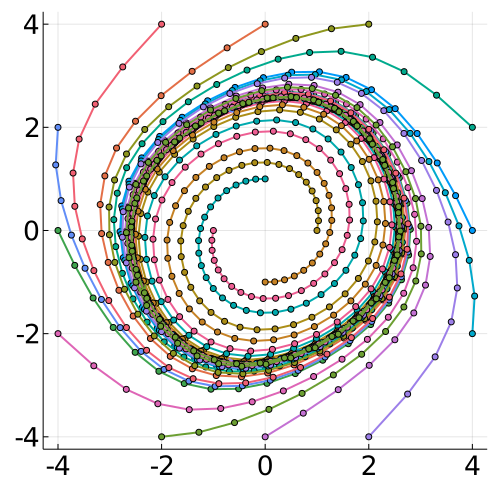}
		\caption{AGDA w/~A=100}
	\end{subfigure}
	\begin{subfigure}[b]{0.245\textwidth}
		\includegraphics[width=\textwidth]{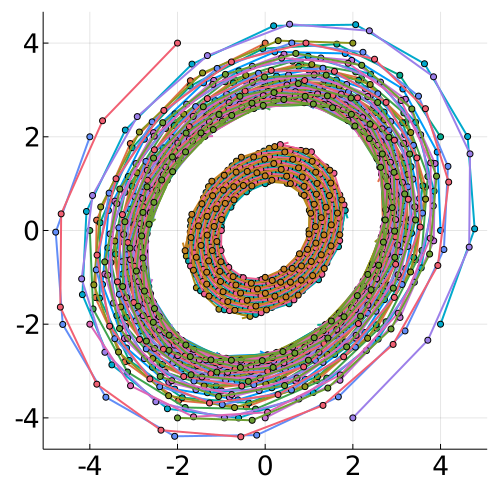}
		\caption{AGDA w/~A=1000}
	\end{subfigure}
	\caption{Sample paths of PPM, EGM, GDA, and AGDA extending Figure~\ref{fig:sample-path}.}
	\label{fig:path-more}
\end{figure}

Figure~\ref{fig:path-more} plots more solution paths of four common first-order methods for minimax problem for solving the two-dimensional minimax problem previously considered in Figure~\ref{fig:sample-path}. This problem is globally $\rho=20$-weakly convex and $\beta=172$-smooth on the box $[-4,4]\times [-4,4]$.

Plots (a)-(d) show the behavior of the Proximal Point Method (PPM)~\eqref{eq:saddle-PPM} with $\eta=2\rho = 40$ and $\lambda=1$. These figures match the landscape described by our theory: $A=1$ is small enough to have local convergence to four different stationary points (each around $\{\pm 2\}\times \{\pm 2\}$), $A=10$ has moderate size and every sample path is attracted into a limit cycle, and finally $A=100$ and $A=1000$ have a globally attractive stationary point.
Plots (e)-(h) show the behavior of the Extragradient Method (EG), which is defined by
\begin{align}\label{eq:EG}
\begin{bmatrix} \tilde x \\ \tilde y \end{bmatrix} &= \begin{bmatrix} x_{k} \\ y_{k} \end{bmatrix} + s\begin{bmatrix} -\nabla_x L(x_k,y_k) \\ \nabla_y L(x_k,y_k) \end{bmatrix}\nonumber\\
\begin{bmatrix} x_{k+1} \\ y_{k+1} \end{bmatrix} &= \begin{bmatrix} x_{k} \\ y_{k} \end{bmatrix} + s\begin{bmatrix} -\nabla_x L(\tilde x,\tilde y) \\ \nabla_y L(\tilde x,\tilde y) \end{bmatrix}
\end{align}
with stepsize chosen as $s = 1/2(\beta+A)$. These figures show that the extragradient method follows the same general trajectory as described by our theory for the proximal point method. For small $A=1$, local convergence occurs. For moderate sized $A=10$ and $A=100$, the algorithm falls into an attractive limit cycle, never converging. For large enough $A=1000$, the method globally converges to a stationary point.
Plots (i)-(l) show the behavior of Gradient Descent Ascent (GDA)~\eqref{eq:GDA} with $s = 1/2(\beta+A)$. This method is known to be unstable and diverge even for convex-concave problems. The same behavior carries over to our nonconvex-nonconcave example, falling into a limit cycle with an increasingly large radius as $A$ grows.
Plots (m)-(p) show the behavior of Alternating Gradient Descent Ascent (AGDA)
\begin{align}\label{eq:AGDA}
x_{k+1} &= x_k -s\nabla_x L(x_k,y_k) \nonumber\\
y_{k+1} &= y_k +s\nabla_y L(x_{k+1},y_k)
\end{align}
with $s = 1/2(\beta+A)$. Again for small $A$, we still see local convergence, but for larger $A=10,100,1000$, AGDA always falls into a limit cycle.


\section{Convex-Concave Optimization Analysis Proofs} \label{app:proofs}

\subsection{Proof of Lemma~\ref{lem:helper0}}
Observe that the assumed strong convexity and strong concavity ensures
\begin{align*}
M(x',y') \leq &M(x,y') -\nabla_x M(x',y')^T(x-x')- \frac{\mu}{2}\|x-x'\|^2\\
&\leq M(x,y) + \nabla_y M(x,y)^T(y'-y) -\nabla_x M(x',y')^T(x-x') - \frac{\mu}{2}\|y-y'\|^2 - \frac{\mu}{2}\|x-x'\|^2
\end{align*}
and symmetrically,
\begin{align*}
M(x',y') &\geq M(x',y) - \nabla_y M(x',y')^T(y-y') + \frac{\mu}{2}\|y-y^*\|^2\\
&\geq M(x,y) + \nabla_x M(x,y)^T(x'-x) - \nabla_y M(x',y')^T(y-y') + \frac{\mu}{2}\|x-x'\|^2 + \frac{\mu}{2}\|y-y'\|^2.
\end{align*}
Combining the above two inequalities gives the first claimed inequality
\begin{equation*}
\mu\left\|\begin{bmatrix} x-x' \\ y-y' \end{bmatrix}\right\|^2 \leq \left(\begin{bmatrix} \nabla_x M(x,y)  \\-\nabla_y M(x,y) \end{bmatrix} -\begin{bmatrix} \nabla_x M(x',y')  \\-\nabla_y M(x',y') \end{bmatrix}\right)^T\begin{bmatrix} x-x' \\ y-y' \end{bmatrix}.
\end{equation*}
Furthermore, when $\nabla M(x',y')=0$, Cauchy–Schwarz gives the second inequality.

\subsection{Proof of Theorem~\ref{thm:standard-contraction}}
First we use Lemma~\ref{lem:helper0} to conclude that if the set $S$ is large enough, $M$ must have a stationary point in $S$. Define $B(z,r)=\{z'|\|z-z'\|\le r\}$ as the ball centered as $a$ with radius $r$. 

\begin{lemma} \label{lem:helper1}
	Suppose $M$ is $\mu$-strongly convex-strongly concave in a set $B(x,r)\times B(y,r)$ for some fixed $(x,y)$ and $r\geq 2\|\nabla M(x,y)\|/\mu$, then there exists a stationary point of $M$ in $B((x,y),r/2)$.
\end{lemma}
\begin{proof}
	Consider the following constrained minimax problem $ \min_{x'\in B(x,r)} \max_{y'\in B(y,r)} M(x,y).$
	Since $M(x,y)$ is strongly convex-strongly concave, it must have a unique solution $(x^*,y^*)$.
	The first-order optimality condition for $(x^*,y^*)$ ensures
	$$ \nabla_x M(x^*,y^*) =-\lambda(x^*-x) \ \ \ \text{ and } \ \ \  -\nabla_y M(x^*,y^*) =-\gamma(y^*-y)$$
	for some constants $\lambda,\gamma\geq 0$ that are nonzero only if $x^*$ or $y^*$ are on the boundary of $B(x,r)$ and $B(y,r)$ respectively. Taking an inner product with $(x^*-x,y^*-y)$ gives
	\begin{equation}\label{eq:lemA2.1}
	\begin{bmatrix} \nabla_x M(x^*,y^*)  \\-\nabla_y M(x^*,y^*) \end{bmatrix}^T\begin{bmatrix} x^*-x  \\ y^*-y \end{bmatrix} =-\left\|\begin{bmatrix}\sqrt{\lambda}(x^*-x)  \\ \sqrt{\gamma}(y^*-y) \end{bmatrix}\right\|^2\leq0.
	\end{equation}
	Applying Lemma~\ref{lem:helper0} and utilizing \eqref{eq:lemA2.1}, we conclude that
	\begin{equation}\label{eq:lemA2.2}
	\mu\left\|\begin{bmatrix} x^*-x  \\ y^*-y \end{bmatrix}\right\|^2 + \begin{bmatrix} \nabla_x M(x,y)  \\-\nabla_y M(x,y) \end{bmatrix}^T\begin{bmatrix} x^*-x  \\ y^*-y \end{bmatrix}\leq 0. 
	\end{equation}
	Hence
	$$ \left\|\begin{bmatrix} x^*-x  \\ y^*-y \end{bmatrix}\right\|^2 \leq \frac{1}{\mu}\left\|\begin{bmatrix} \nabla_x M(x,y)  \\-\nabla_y M(x,y) \end{bmatrix}\right\|\left\|\begin{bmatrix} x^*-x  \\ y^*-y \end{bmatrix}\right\|\ ,$$
	whereby
	$$ \left\|\begin{bmatrix} x^*-x  \\ y^*-y \end{bmatrix}\right\| \leq \frac{1}{\mu}\left\|\begin{bmatrix} \nabla_x M(x,y)  \\-\nabla_y M(x,y) \end{bmatrix}\right\| \leq r/2\ ,$$
	utilizing our assumed condition on $r$.
	Since $(x^*,y^*)$ lies strictly inside $B((x,y),r)$, the first-order optimality condition implies $(x^*,y^*)$ is a stationary point of $M$.
\end{proof}

Lemma~\ref{lem:helper1} ensures the existence of a nearby stationary point $(x^*,y^*)$. Then the standard proof for strongly monotone (from Lemma~\ref{lem:helper0}) and Lipschitz operators gives a contraction whenever $s\in(0, 2\mu/\beta^2) $:
\begin{align*}
\left\|\begin{bmatrix} x_{k+1}-x^* \\ y_{k+1}-y^* \end{bmatrix}\right\|^2
&= \left\|\begin{bmatrix} x_{k}-x^* \\ y_{k}-y^* \end{bmatrix}\right\|^2 - 2s\begin{bmatrix} \nabla_x M(x_k,y_k) \\ -\nabla_y M(x_k,y_k) \end{bmatrix}^T\begin{bmatrix} x_k-x^* \\ y_k-y^* \end{bmatrix} + s^2\left\|\begin{bmatrix} \nabla_x M(x_k,y_k) \\ -\nabla_y M(x_k,y_k) \end{bmatrix}\right\|^2\\
&\leq \left\|\begin{bmatrix} x_{k}-x^* \\ y_{k}-y^* \end{bmatrix}\right\|^2 - 2\mu s\left\|\begin{bmatrix} x_{k}-x^* \\ y_{k}-y^* \end{bmatrix}\right\|^2 + \beta^2s^2\left\|\begin{bmatrix} x_{k}-x^* \\ y_{k}-y^* \end{bmatrix}\right\|^2\\
&= \left(1 - 2\mu s+\beta^2s^2\right)\left\|\begin{bmatrix} x_{k}-x^* \\ y_{k}-y^* \end{bmatrix}\right\|^2 \ , 
\end{align*}
where the inequality utilizes \eqref{eq:lemA2.2} at $(x,y)=(x_k,y_k)$ and the smoothness of $M(x,y)$.

\end{document}